\documentclass[a4paper,10pt,onecolumn,twoside]{article}
\usepackage{mathrsfs}
\usepackage{mathrsfs}
\usepackage{amsfonts}
\usepackage{fancyhdr}
\usepackage{amsmath,amsfonts,amssymb,graphicx,amsthm}
\allowdisplaybreaks

\usepackage{amsmath,latexsym,amsfonts,indentfirst,amsthm,amsxtra,amssymb,bm}

\usepackage{amssymb,mathrsfs,graphicx,enumerate}

\numberwithin{equation}{section}

 \newtheorem{lemma}{Lemma}[section]
 
 \newtheorem{theorem}{Theorem}[section]

 \theoremstyle{remark}
 \newtheorem{remark}{Remark}[section]

\numberwithin{equation}{section}

\begin{document}

\title{\bf Nonlinear stability of rarefaction waves for a viscous radiative and reactive gas with large initial perturbation}
\author{{\bf Guiqiong Gong}\\[1mm]
School of Mathematics and Statistics, Wuhan University, Wuhan 430072, China\\
Computational Science Hubei Key Laboratory, Wuhan University, Wuhan 430072, China\\
Email address:gongguiqiong@whu.edu.cn\\[2mm]
{\bf Lin He}\\[1mm]
College of Mathematics, Sichuan University, Chengdu 610064, China\\
Email address: helin19891021@163.com\\[2mm]
{\bf Yongkai Liao}\footnote{Corresponding author. }\\[1mm]
Institute of Applied Physics and Computational Mathematics, Beijing 100088, China\\
Email address: liaoyongkai@126.com
}

\date{}

\maketitle


\begin{abstract}
We investigate the time-asymptotically nonlinear stability of rarefaction waves to the Cauchy problem of an one-dimensional compressible Navier-Stokes type system for a viscous, compressible, radiative and reactive gas, where the constitutive relations for the pressure $p$, the specific internal energy $e$, the specific volume $v$, the absolute temperature $\theta$, and the specific entropy $s$ are given by $p=R\theta/v +a\theta^4/3$, $e=C_v\theta+av\theta^4$, and $s=C_v\ln \theta+ 4av\theta^3/3+R\ln v$ with $R>0$, $C_{v}>0$, and $a>0$ being the perfect gas constant, the specific heat and the radiation constant, respectively.

For such a specific gas motion, a somewhat surprising fact is that, general speaking, the pressure $\widetilde{p}(v,s)$ is not a convex function of the specific volume $v$ and the specific entropy $s$. Even so, we show in this paper that the rarefaction waves are time-asymptotically stable for large initial perturbation provided that the radiation constant $a$ and the strength of the rarefaction waves are sufficiently small. The key point in our analysis is to deduce the positive lower and upper bounds on the specific volume and the absolute temperature, which are uniform with respect to the space and the time variables, but are independent of the radiation constant $a$.

\end{abstract}

\section{Introduction}
In this paper, we investigate the large-time behavior of global, strong, large-amplitude solutions to the Cauchy problem of an one-dimensional compressible Navier-Stokes type system for a viscous radiative and reactive gas. The model is described as follows (cf. \cite{Ducomet-M3AS-1999, Liao-Zhao-CMS-2017, Liao-Zhao-JDE-2018, Umehara-Tani-JDE-2007}):
\begin{eqnarray}\label{1.1}
    v_t-u_x&=&0,\nonumber\\
    u_t+p\left(v,\theta\right)_x&=&\left(\frac{\mu u_x}{v}\right)_x,\nonumber\\
    \left(e+\frac{u^2}{2}\right)_t+(up(v,\theta))_x&=&\left(\frac{\mu u u_{x}}{v}\right)_x+\left(\frac{\kappa\left(v,\theta\right)\theta_{x}}{v}\right)_{x}+\lambda\phi z,\\
    z_{t}&=&\left(\frac{dz_x}{v^{2}}\right)_{x}-\phi z.\nonumber
\end{eqnarray}
Here $x\in\mathbb{R}$ is the Lagrangian space variable, $t\in\mathbb{R}^+$ the time variable. The unknown quantities are the specific volume $v=v\left(t,x\right)$, the velocity\,$u=u\left(t,x\right)$, the absolute temperature\, $\theta=\theta\left(t,x\right)$, and the mass fraction of the reactant\, $z=z\left(t,x\right)$. The positive constants $d$ and $\lambda$ stand for the species diffusion coefficient and the difference in the heat between the reactant and the product, respectively. According to the Arrhenius law \cite{Ducomet-Zlotnik-NonliAnal-2005,Umehara-Tani-JDE-2007}, the reaction rate function $\phi=\phi\left(\theta\right)$ is given by

\begin{equation}\label{1.2}
 \phi\left(\theta\right)=K\theta^{\beta}\exp\left(-\frac{A}{\theta}\right),
\end{equation}
where positive constants $K$ and $A$ represent the coefficients of the rates of the reactant and the activation energy, respectively. Besides, $\beta$ is a non-negative number.

Due to the Stefan-Boltzmann radiative law \cite{Mihalas-Mihalas-1984, Umehara-Tani-JDE-2007}, the pressure $p$ and the specific internal energy $e$ consist of a fourth-order term radiative part in the absolute temperature $\theta$ as well as the perfect polytropic contribution
\begin{equation}\label{1.3}
  p\left(v,\theta\right)=\frac{R\theta}{v}+\frac{a\theta^{4}}{3}, \quad e(v,\theta)=C_{v}\theta+av\theta^{4},
\end{equation}
where the positive constants $R$ and $C_{v}$ are the perfect gas constant and the specific heat capacity at constant volume, respectively. Specifically,  as shown in \cite{Liao-NARWA-2020,Mihalas-Mihalas-1984}, $C_{v}=\frac{3}{2}R$ for the radiative gas. $a>0$ is the radiation constant which measures the amount of heat that is emitted by a blackbody, which absorbs all of the radiant energy that hits it, and will emit all the radiant energy. Moreover, we have (cf. \cite{Liao-NARWA-2020,Mihalas-Mihalas-1984})
\begin{equation}\label{Definition of a}
    a=\frac{4\sigma}{c}=\frac{8\pi^{5}k_{B}^{4}}{15c^{3}h^{3}},
\end{equation}
where $\sigma$ is the Stefan-Boltzmann constant, $c$ is the speed of light, $k_{B}$ is the Boltzmann constant, and $h$ is the Planck's constant. Numerically, $a=7.5657\times 10^{-16}\text{J}\text{m}^{-3}\text{K}^{-4}$. In general, compared with that of the perfect gas constant $R$ and the specific heat $C_{v}$, the radiation constant $a$ is much smaller.

On the other hand, one can conclude from \eqref{1.3} and the second law of thermodynamics that
\begin{equation}\label{entropy}
s(v,\theta)=C_v\ln \theta+\frac 43av\theta^3+R\ln v.
\end{equation}
If one takes $a=0$, then the above constitutive relations for the five thermodynamical variables $p, v, \theta, s$ and $e$ given by \eqref{1.3} and \eqref{entropy} reduce to the equations of state for ideal polytropic gas. If $a>0$, and we choose $(v,\theta)$ or $(v,s)$ as independent variables and write $(p,e,s)=(p(v,\theta), e(v,\theta),s(v,\theta))$ or $(p,e,\theta) =(\tilde{p}(v,s),\tilde{e}(v,s),\tilde{\theta}(v,s))$, respectively,
then after cumbersome calculations, we can deduce that (see \cite{Liao-NARWA-2020} for details )

\begin{eqnarray}\label{p7}
\frac{\partial^2\tilde{p}(v,s)}{\partial v^2}&=&\frac{1}{s^{3}_{\theta}}\left[\frac{C_{v}R^{3}
+3C_{v}^{2}R^{2}+2C_{v}^{3}R}{v^{3}\theta^{2}} +\frac{\left(40aC_{v}R^{2}+28aC_{v}^{2}R-8aR^{3}\right)\theta}{v^{2}}\right.\nonumber\\
&&\left.+\frac{\left(496a^{2}C_{v}R+192a^{2}R^{2}\right)\theta^{4}}{3v}+\frac{\left(640a^{3}C_{v}+7488a^{3}R\right)\theta^{7}}{27}
+\frac{1792a^{4}v\theta^{10}}{27}\right],
\end{eqnarray}
\begin{eqnarray}\label{p8}
\frac{\partial^2\tilde{p}(v,s)}{\partial s^2}=\frac{(\tilde{\theta}_{s})^{2}}{s_{\theta}} \left[\frac{C_{v}R}{v\theta^{2}}
+\left(\frac{16aC_{v}}{3}-8aR\right)\theta +\frac{16a^{2}v\theta^{4}}{3}\right],
\end{eqnarray}
and
\begin{eqnarray}\label{p9}
&&\frac{\partial^2\tilde{p}(v,s)}{\partial s^2} \frac{\partial^2\tilde{p}(v,s)}{\partial v^2}-\left(\frac{\partial^2\tilde{p}(v,s)}{\partial v\partial s}\right)^2\nonumber\\
&=&\frac{(\tilde{\theta}_{s})^{2}}{s^{2}_{\theta}} \left[\frac{C_{v}R^{3}+C_{v}^{2}R^{2}}{\theta^{2}v^{4}}
+\frac{\left(32aC_{v}^{2}R-52aC_{v}R^{2}-24aR^{3}\right)\theta}{3v^{3}}\right.\\
&&\left.+\frac{\left(448a^{2}C_{v}R-1200a^{2}R^{2}\right)\theta^{4}}{9v^{2}}
-\frac{320a^{3}R\theta^{7}}{9v}-\frac{256a^{4}\theta^{10}}{9}\right].\nonumber
\end{eqnarray}

From \eqref{p7}, \eqref{p8}, and \eqref{p9}, it is easy to see that $\tilde{p}(v,s)$ is a convex function of $v$ and $s$ for the ideal polytropic gas, while if $a>0$, it is not clear whether $\tilde{p}(v,s)$ is a convex function of $v$ and $s$ or not.

As in \cite{Liao-Zhao-JDE-2018, Umehara-Tani-JDE-2007, Umehara-Tani-PJA-2008}, we also assume that the bulk viscosity $\mu$ is a positive constant and the thermal conductivity $\kappa=\kappa\left(v,\theta\right)$ takes the form
\begin{equation}\label{1.4}
\kappa\left(v,\theta\right)=\kappa_{1}+\kappa_{2}v\theta^{b},
\end{equation}
where $\kappa_{1}$, $\kappa_{2}$, and $b$ are both positive constants. Furthermore, the system \eqref{1.1} is supplemented with the initial data
\begin{equation}\label{1.5}
 \left(v\left(0,x\right),u\left(0,x\right),\theta\left(0,x\right), z\left(0,x\right)\right)=\left(v_0\left(x\right),u_0\left(x\right),\theta_{0}\left(x\right), z_{0}\left(x\right)\right)
\end{equation}
for $x\in\mathbb{R}$, which is assumed to satisfy the far-field condition:
\begin{equation}\label{1.6}
\lim_{|x|\rightarrow\infty}\left(v_0\left(x\right),u_0\left(x\right),\theta_{0}\left(x\right), z_{0}\left(x\right)\right)=(v_\pm,u_\pm,\theta_\pm,0).
 \end{equation}
Here $v_{\pm}>0$, $u_{\pm}$ and $\theta_{\pm}>0$ are prescribed constants.

The problem on the global solvability and the precise description of the large-time behaviors of the global solutions constructed to the initial value problem and the initial-boundary value problems of system \eqref{1.1}, \eqref{1.2}, \eqref{1.3}, \eqref{1.4}, \eqref{1.5} is a hot topic in the filed of nonlinear partial differential equations and many results have been obtained recently. A complete literature in this direction is beyond the scope of this paper and to go directly to the main points of the present paper, in what follows we only review some former results which are closely related to our main results.

\begin{itemize}
\item For multidimensional case, the global existence, uniqueness and exponential stability of spherically (cylindrically) symmetric solutions in a bounded concentric annular domain were studied in \cite{Qin-Zhang-Su-Cao-JMFM-2016, Umehara-Tani-AA-2008, Zhang-Non-2017}. Recently, \cite{Liao-Wang-Zhao-2018} investigated the global solvability and the precise description of the large time behavior of the global solutions constructed in an exterior domain. Here, the asymptotics of the global solutions constructed in \cite{Qin-Zhang-Su-Cao-JMFM-2016, Umehara-Tani-AA-2008, Zhang-Non-2017} and \cite{Liao-Wang-Zhao-2018}, as in \cite{Jiang-ZHeng-JMP-2012, Jiang-ZHeng-ZAMP-2014} and \cite{He-Liao-Wang-Zhao-2017, Liao-Zhao-JDE-2018}, are constant equilibrium states $(v_\infty, u_\infty,\theta_\infty,0)$ of \eqref{1.1} satisfying $v_\infty>0, \theta_\infty>0$, which are uniquely determined by the initial data for the corresponding initial-boundary value problem in bounded domain and by the far fields of the initial data for the case in an exterior domain;

\item For the one-dimensional initial-boundary value problem in the interval $[0,1]$, the existence and uniqueness of global classical solutions was established in \cite{Ducomet-M3AS-1999} for the following initial-boundary value problem
\begin{eqnarray}\label{IBVP-I}
 \left(v\left(0,x\right),u\left(0,x\right),\theta\left(0,x\right), z\left(0,x\right)\right)&=&\left(v_0\left(x\right),u_0\left(x\right),\theta_{0}\left(x\right), z_{0}\left(x\right)\right),\quad x\in(0,1),\nonumber\\
 u(t,x)&=&0,\quad x=0,1,\quad t\geq 0,\\
 \left(\theta_x(t,x),z_x(t,x)\right)&=&(0,0), \quad x=0,1,\quad  t\geq 0,\nonumber
\end{eqnarray}
while for the initial-boundary value problem
\begin{eqnarray}\label{IBVP-II}
 \left(v\left(0,x\right),u\left(0,x\right),\theta\left(0,x\right), z\left(0,x\right)\right)&=&\left(v_0\left(x\right),u_0\left(x\right),\theta_{0}\left(x\right), z_{0}\left(x\right)\right),\quad x\in(0,1),\nonumber\\
 \sigma(t,x)\equiv -p(v(t,x),\theta(t,x))+\frac{\mu u_x(t,x)}{v(t,x)}&=&-p_e,\quad x=0,1,\quad t\geq 0,\\
 \left(\theta_x(t,x),z_x(t,x)\right)&=&(0,0), \quad x=0,1,\quad  t\geq 0\nonumber
\end{eqnarray}
    for some positive constant $p_e>0$, similar global solvability results were obtained in \cite{Liao-Zhao-CMS-2017, Qin-Hu-Wang-Huang-Ma-JMAA-2013, Umehara-Tani-JDE-2007, Umehara-Tani-PJA-2008}. Moreover, it is shown in \cite{Jiang-ZHeng-JMP-2012} and \cite{Jiang-ZHeng-ZAMP-2014} that the asymptotics of the global solutions constructed above can be exactly described by
    $\left(1, 0, \theta_\infty, 0\right)$ with $\theta_\infty$ being a positive constant uniquely determined by
    $$
    C_v\theta_\infty+a\theta^4_\infty=\int^1_0\left(\frac 12\left|u_0(x)\right|^2+C_v\theta_0(x)+av_0(x)\left|\theta_0(x)\right|^4+\lambda z_0(x)\right)dx
    $$ for the initial-boundary value problem \eqref{1.1}, \eqref{1.2}, \eqref{1.3}, \eqref{1.4}, \eqref{1.5}, \eqref{IBVP-I} and $\left(v_\infty, 0, \theta_\infty, 0\right)$ with $v_\infty$ and $\theta_\infty$ are positive constants uniquely determined by
    \begin{eqnarray*}
    \frac{R\theta_\infty}{v_\infty}+\frac a3\theta^4_\infty&=&p_e,\\
    C_v\theta_\infty+av_\infty\theta^4_\infty p_ev_\infty&=&\int^1_0\left(\frac 12\left|u_0(x)\right|^2+C_v\theta_0(x)+av_0(x)\left|\theta_0(x)\right|^4+\lambda z_0(x)+p_ev_0(x)\right)dx
    \end{eqnarray*}
     for the initial-boundary value problem \eqref{1.1}, \eqref{1.2}, \eqref{1.3}, \eqref{1.4}, \eqref{1.5}, \eqref{IBVP-II}, respectively. Note that since $\int^1_0 v(t,x)ds$ is conserved for the initial-boundary value problem \eqref{1.1}, \eqref{1.2}, \eqref{1.3}, \eqref{1.4}, \eqref{1.5}, \eqref{IBVP-I}, while $\int^1_0u(t,x)dx$ is conserved for the initial-boundary value problem \eqref{1.1}, \eqref{1.2}, \eqref{1.3}, \eqref{1.4}, \eqref{1.5}, \eqref{IBVP-II}, respectively, one can thus assume without loss of generality that $\int^1_0v_0(x)dx=1$ for the initial-boundary value problem \eqref{1.1}, \eqref{1.2}, \eqref{1.3}, \eqref{1.4}, \eqref{1.5}, \eqref{IBVP-I} and $\int^1_0u_0(x)dx=0$ for the initial-boundary value problem \eqref{1.1}, \eqref{1.2}, \eqref{1.3}, \eqref{1.4}, \eqref{1.5} and \eqref{IBVP-II};

\item For the Cauchy problem \eqref{1.1}, \eqref{1.2}, \eqref{1.3}, \eqref{1.4}, \eqref{1.5}, \eqref{1.6}, the existence of a unique global solution was established very recently in \cite{Liao-Zhao-JDE-2018,Liao-ActaMSci-2019} for the case when the far fields $(v_\pm, u_\pm, \theta_\pm)$ of the initial data $(v_0(x), u_0(x), \theta_0(x))$ are equal, i.e., $(v_-,u_-,\theta_-)=(v_+, u_+,\theta_+)$. See also \cite{He-Liao-Wang-Zhao-2017} for the case with temperature-dependent viscosity and \cite{Liao-Xu-Zhao-SCM-2018} for the case with density-dependent viscosity. Here since $(v_-,u_-,\theta_-)=(v_+, u_+,\theta_+)$, the asymptotics of the global solutions constructed in \cite{He-Liao-Wang-Zhao-2017, Liao-Zhao-JDE-2018} are exactly the far fields $(v_\pm, u_\pm,\theta_\pm, 0)$ of the initial data $(v_0(x), u_0(x), \theta_0(x), z_0(x))$. The asymptotic stability of 1-rarefaction wave to the system $\eqref{1.1}_{1}$-$\eqref{1.1}_{3}$ ($z=0$), \eqref{1.2}, \eqref{1.3}, \eqref{1.5}, and \eqref{1.6} without viscosity ($\mu=0$) under the small perturbation was studied in \cite{Li-Wang-Yang-2020}. Recently, Liao \cite{Liao-NARWA-2020} have studied nonlinear stability of rarefaction waves for the system \eqref{1.1}, \eqref{1.2}, \eqref{1.3}, \eqref{1.4}, \eqref{1.5}, and \eqref{1.6} when the viscosity $\mu$ takes the following form:
\begin{eqnarray}\label{1.40}
\mu=\mu\left(v,\theta\right)=h\left(v\right)\theta^{\alpha},\quad h\left(v\right)\sim
 \begin{cases}
 v^{-\ell_{1}}, & $$ v\rightarrow 0^{+} $$,\\
 v^{\ell_{2}}, & $$ v\rightarrow \infty $$,
 \end{cases}
 \quad v\left|h'(v)\right|^{2}\leq Ch^{3}(v).
\end{eqnarray}
Here $h(v)$ is a smooth function of $v$ for $v>0$ and $\alpha$, $\ell_{1}$, and $\ell_{2}$ are positive constants. It should be pointed out that \eqref{1.40} can not cover the case when $\mu$ is a positive constant even when $\alpha$ goes to zero.

\end{itemize}

The main purpose of this manuscript is to study the nonlinear stability of rarefaction waves for the system \eqref{1.1}, \eqref{1.2}, \eqref{1.3}, \eqref{1.4}, \eqref{1.5}, and \eqref{1.6} with constant viscosity ($\mu\equiv C$) under the large initial perturbation.
For the Cauchy problem \eqref{1.1}, \eqref{1.2}, \eqref{1.3}, \eqref{1.4}, \eqref{1.5}, \eqref{1.6}, if the far fields $(v_\pm, u_\pm,\theta_\pm)$ of the initial data $(v_0(x), u_0(x), \theta_0(x))$ are not equal, i.e., $(v_-,u_-,\theta_-)\not=(v_+,u_+,\theta_+)$, the asymptotics of the global solutions should be nontrivial and are expected to be described by the unique global entropy solution $\left(V^r(x/t), U^r(x/t), \Theta^r(x/t),0\right)$ of the resulting Riemann problem of the corresponding compressible Euler equations
\begin{eqnarray}\label{1.10}
v_t - u_x &=&0,\nonumber\\
u_t+p(v,\theta)_x&=&0,\nonumber\\
\left(e+\frac {u^2}{2}\right)_t+(up(v,\theta))_x&=&0,\\
z_t&=&0\nonumber
\end{eqnarray}
with Riemann data
\begin{equation}\label{1.11}
(v(0,x),u(0,x),\theta(0,x),z(0,x))=\left(v_{0}^{r}(x),u_{0}^{r}(x),\theta_{0}^{r}(x),z_{0}^{r}(x)\right)=
\left\{
\begin{array}{rl}
(v_-,u_-,\theta_-,0),& x<0,\\[1mm]
(v_+,u_+,\theta_+,0),& x>0.
\end{array}
\right.
\end{equation}
In fact , it is expected, cf. \cite{Kawashima-Matsmura-CMP-1985, Kawashima-Matsmura-Nishihara-PJA-1986, Liu-CPAM-1986, Liu-Xin-CMP-1988, Liu-Xin-AJM-1997, Liu-Zeng-MAMS-1997, Matsmura-Nishihara-JJAM-1986, Smoller-1999} and the references cited therein, that if the unique global entropy solution $\left(V^r(x/t), U^r(x/t), \Theta^r(x/t), 0\right)$ of the Riemann problem \eqref{1.10}, \eqref{1.11} consists of rarefaction waves $(V^{R_i}(x/t), U^{R_i}(x/t), \Theta^{R_i}(x/t),0)$ of the $i-$th family ($i=1,3$), shock waves $(V^{S_i}(x/t), U^{S_i}(x/t), \Theta^{S_i}(x/t),$ $0)$ of the $i-$th family ($i=1,3$), contact discontinuity $(V^{CD}(x/t), U^{CD}(x/t), \Theta^{CD}(x/t), 0)$ of the second family, and/or their superpositions, then the large time behavior of the global solution $(v(t,x), u(t,x),\theta(t,x), z(t,x))$ of the Cauchy problem \eqref{1.1}, \eqref{1.2}, \eqref{1.3}, \eqref{1.4}, \eqref{1.5}, \eqref{1.6} is expected to be well-described by the rarefaction wave $(V^{R_i}(x/t), U^{R_i}(x/t), \Theta^{R_i}(x/t),0)$ of the $i-$th family ($i=1,3$), viscous shock profile $(V^{VSW_i}(x-s_it), U^{VSW_i}(x-s_it), \Theta^{VSW_i}(x-s_it),0)$ of the $i-$th family ($i=1,3$) under suitable shift, viscous contact discontinuity wave $(V^{VCD}(t,x), U^{VCD}(t,x),$ $\Theta^{VCD}(t,x), 0)$ of the second family, and/or their superpositions.

As in \cite{Duan-Liu-Zhao-TAMS-2009,Liao-NARWA-2020}, it will be convenient to consider the following equations for the entropy $s$ and the absolute temperature $\theta$:
\begin{equation}\label{1.7}
s_t=\left(\frac{\kappa(v,\theta)\theta_x}{v\theta}\right)_x+\frac{\kappa(v,\theta)\theta^2_x}{v\theta^2}
+\frac{\mu u_x^2}{v\theta}+\frac{\lambda\phi z}{\theta},
\end{equation}
and
\begin{equation}\label{1.8}
\theta_t+\frac{\theta p_{\theta}u_{x}}{e_{\theta}}=\frac{1}{e_{\theta}}\left(\frac{\kappa(v,\theta)\theta_x}{v}\right)_x+\frac{\mu u_x^2}{ve_{\theta}}+\frac{\lambda\phi z}{e_{\theta}}.
\end{equation}
Notice that $p_{\theta}:=\frac{\partial p(v,\theta)}{\partial\theta}=\frac Rv+\frac 43a\theta^3$ and $e_{\theta}:=\frac{\partial e(v,\theta)}{\partial\theta}=C_v+4av\theta^3$.

From now on, we will consider (\ref{1.1})$_1$, (\ref{1.1})$_2$, (\ref{1.7}), (\ref{1.1})$_4$ with the initial data
\begin{equation}\label{1.9}
(v(t,x),u(t,x),s(t,x),z(t,x))|_{t=0}=(v_0(x),u_0(x),s_0(x),z_{0}(x))\to (v_\pm,u_\pm,s_\pm,0)  \quad \textrm{as}\  x \to \pm\infty.
\end{equation}
Here $v_{\pm}>0$, $u_{\pm}$, $s_{\pm}:=C_v\ln \theta_\pm+\frac 43av_\pm\theta_\pm^3+R\ln v_\pm$ are constants and $s_0(x):=C_v\ln \theta_0(x)+\frac 43av_0(x)\theta_0(x)^3+R\ln v_0(x)$. Moreover, we assume that $s_+=s_-=\bar{s}$ for considering the expansion waves to (\ref{1.1}).

It is well known that the equations (\ref{1.1}) can be approximated by the Riemann problem of the following equations:

\begin{eqnarray}\label{1.10a}
             v_t - u_x &=&0, \nonumber\\
            u_t+({\tilde{p}}(v,s))_x&=&0 ,\nonumber\\
             s_t&=&\frac{\lambda\phi z}{\theta},\\
             z_t&=&-\phi z,\nonumber
\end{eqnarray}
with Riemann data
\begin{equation}\label{1.11a}
(v(0,x),u(0,x),s(0,x),z(0,x))=(v_{0}^{R}(x),u_{0}^{R}(x),s_{0}^{R}(x),z_{0}^{R}(x))=
\left\{
\begin{array}{rl}
(v_-,u_-,s_-,0),& x<0,\\
(v_+,u_+,s_+,0),& x>0.
\end{array}
\right.
\end{equation}

The solutions of the Riemann problem (\ref{1.10a})-(\ref{1.11a}) have two characteristics which leads to two families of expansion (rarefaction) waves: the 1-rarefaction wave $\left(V^R_1(\frac{x}{t}),~U^R_1(\frac{x}{t}),~\bar{s},0\right)$ and the 3-rarefaction wave
$\left(V^R_3(\frac{x}{t}),~U^R_3(\frac{x}{t}),~\bar{s},0\right)$. We define the regime
\begin{eqnarray*}
\mathbb{R}_1(v_-, u_-,\bar{s},0)&=&\left\{(v, u, s,z)| u=u_-+\displaystyle\int_{v_-}^v \sqrt{-\tilde{p}_{\xi}(\xi,\bar{s})}\mathrm{d}\xi, u\geq u_-, s=\bar{s}, z=0 \right\},\\
\mathbb{R}_3(v_m, u_m,\bar{s},0)&=&\left\{(v, u, s,z)| u=u_m-\displaystyle\int_{v_m}^v \sqrt{-\tilde{p}_{\xi}(\xi,\bar{s})}\mathrm{d}\xi, u\geq u_m, s=\bar{s}, z=0 \right\},
\end{eqnarray*}
and further assume that there exists a unique constant state $(v_m, u_m)\in \mathbb{R}^2 (v_m >0)$, which satisfies $(v_m, u_m)\in \mathbb{R}_1 (v_-, u_-)$ and $(v_+, u_+) \in \mathbb{R}_3 (v_m, u_m)$. Then the unique weak solution
$\left(V^R(\frac{x}{t}),U^R(\frac{x}{t}),S^R(\frac{x}{t}),0\right)$ to the system (\ref{1.10a})-(\ref{1.11a}) is characterized by
\begin{equation}\label{1.12}
\left(V^R\left(\frac{x}{t}\right), U^R\left(\frac{x}{t}\right), S^R\left(\frac{x}{t}\right),0\right)\\
=\left(V^R_1\left(\frac{x}{t}\right)+V^R_3\left(\frac{x}{t}\right)-v_m,
~U^R_1\left(\frac{x}{t}\right)+U^R_3\left(\frac{x}{t}\right)-u_m,~\bar{s},0\right),
\end{equation}
with $(V_i^R(\frac{x}{t}),~U_i^R(\frac{x}{t}),~S^R(\frac{x}{t}),0) (i=1,3)$ satisfy the following equations:
\begin{eqnarray}\label{1.13}
    S^R\left(\frac{x}{t}\right)&=&\bar{s},\nonumber\\
   U_1^R\left(\frac{x}{t}\right)-\displaystyle\int_1^{V_1^R\left(\frac{x}{t}\right)}\sqrt{-\tilde{p}_{\xi}(\xi,\bar{s})} \mathrm{d}\xi&=&u_--
   \displaystyle\int^{v_-}_1\sqrt{-\tilde{p}_{\xi}(\xi,\bar{s})}\mathrm{d}\xi,\nonumber\\
  \lambda_{1x}\left(V_1^R\left(\frac{x}{t}\right),\bar{s}\right)&>&0,\nonumber\\
  \lambda_1(v,s)&=&-\sqrt{-\tilde{p}_{v}(v,s)},\\
  U_3^R\left(\frac{x}{t}\right)+\displaystyle\int_1^{V_3^R\left(\frac{x}{t}\right)} \sqrt{-\tilde{p}_{\xi}(\xi,\bar{s})}\mathrm{d}\xi &=&u_m+\displaystyle\int_1^{v_m}\sqrt{-\tilde{p}_{\xi}(\xi,\bar{s})}\mathrm{d}\xi,\nonumber\\
  \lambda_{3x}\left(V_3^R\left(\frac{x}{t}\right),\bar{s}\right)&>&0,\nonumber\\
  \lambda_3(v,s)&=&\sqrt{-\tilde{p}_{v}(v,s)}.\nonumber
\end{eqnarray}

To construct the approximate waves $(V(t,x), U(t,x), S(t,x),0)$, we begin with the following Burger's equation (cf. \cite{Matsmura-Nishihara-JJAM-1986}). Let $\omega_i(t,x)$  $(i=1,3)$ be the unique global smooth solution to the Cauchy problem
\begin{eqnarray}\label{1.14}
 \omega_{it}+\omega_{i}\omega_{ix}&=&0,\\
 \omega_i(t,x)|_{t=0}&=&\omega_{i0}(x)=\frac{w_{i+}+w_{i-}}{2}+\frac{w_{i+}-w_{i-}}{2}K_q\displaystyle\int^{\epsilon x}_0(1+y^2)^{-q}\mathrm{d}y,\nonumber
\end{eqnarray}
where $q>\frac{3}{2}$, $K_q=\left(\int^{+\infty}_0(1+y^2)^{-q}\mathrm{d}y\right)^{-1}, \epsilon >0$ is a positive constant to be determined later, and
\begin{eqnarray*}
  \omega_{1-}&=&\lambda_1(v_-,\bar{s})=-\sqrt{-\tilde{p}_{v}(v_{-},\bar{s})},\\
  \omega_{1+}&=&\lambda_1(v_m,\bar{s})=-\sqrt{-\tilde{p}_{v}(v_{m},\bar{s})},\\
  \omega_{3-}&=&\lambda_3(v_m,\bar{s})=\sqrt{-\tilde{p}_{v}(v_{m},\bar{s})},\\
  \omega_{3+}&=&\lambda_3(v_+,\bar{s})=\sqrt{-\tilde{p}_{v}(v_{+},\bar{s})}.
\end{eqnarray*}
Then, by setting $\epsilon=\delta=|v_--v_+|+|u_--u_+|$, the approximate rarefaction waves $(V(t,x), U(t,x), S(t,x),0)$ is defined by
\begin{eqnarray}\label{1.15}
&&\left(V(t,x),U(t,x),S(t,x),0\right)\\
&=& \left(V_1(t+1, x)+V_3(t+1, x)-v_m, U_1(t+1,x)+U_3(t+1, x)-u_m, \bar{s},0\right),\nonumber
\end{eqnarray}
where $(V_i(t,x), U_i(t,x))(i=1,3)$ satisfy
\begin{eqnarray}\label{1.16}
  \lambda_i(V_i(t,x),\bar{s})&=&\omega_i(t,x),i=1,3,\nonumber\\
  \lambda_1(v,s)&=&-\sqrt{-\tilde{p}_{v}(v,s)},\nonumber\\
  \lambda_3(v,s)&=&\sqrt{-\tilde{p}_{v}(v,s)},\\
  U_1(t,x)&=&u_-+{\displaystyle\int^{V_1(t,x)}_{v_-}}\sqrt{-\tilde{p}_{\xi}(\xi,\bar{s})}\mathrm{d}\xi,\nonumber\\
  U_3(t,x)&=&u_m-{\displaystyle\int^{V_3(t,x)}_{v_m}}\sqrt{-\tilde{p}_{\xi}(\xi,\bar{s})}\mathrm{d}\xi,\nonumber
\end{eqnarray}
and $\Theta(t,x)$ is given by
\begin{eqnarray*}
\Theta(t,x)=\tilde{\theta}(V(t,x),\bar{s}).
\end{eqnarray*}

Furthermore, if we denote the strength of the rarefaction waves by
\begin{eqnarray*}
\delta=|v_--v_+|+|u_--u_+|,
\end{eqnarray*}
then our main result is the following stability theorem.
\begin{theorem}\label{Th1.1}
Suppose that
\begin{itemize}
\item The parameters $b$ and $\beta$ are assumed to satisfy
\begin{equation*}
b>6, \quad 0\leq\beta< b+3;
 \end{equation*}
\item There exist positive constants $0<\underline{V}\leq 1$, $\overline{V}>1$, $0<\underline{\Theta}\leq 1$, $\overline{\Theta}>1$, which do not depend on the strength of the rarefaction wave $\delta$ and the radiation constant $a$, such that
\begin{eqnarray*}
2\underline{V}&\leq& v_0(x)\leq \frac{1}{2}\overline{V},\\
2\underline{V}&\leq& V(t,x)\leq \frac{1}{2}\overline{V},\\
2\underline{\Theta} &\leq& \theta_0(x) \leq \frac{1}{2}\overline{\Theta},\\
2\underline{\Theta} &\leq& \Theta(t,x) \leq \frac{1}{2}\overline{\Theta}
\end{eqnarray*}
hold for all $(t,x) \in \mathbb{R}_+\times\mathbb{R}$,
\begin{eqnarray*}
  \left(v_{0}(x)-V(0,x), u_{0}(x)-U(0,x), \theta_{0}(x)-\Theta(0,x), z_{0}(x)\right)\in H^{1}\left(\mathbb{R}\right),\nonumber\\
   \frac{\partial^2\left(u_{0}(x)-U(0,x)\right)}{\partial x^2}\in L^{2}\left(\mathbb{R}\right),\quad z_{0}(x)\in L^{1}\left(\mathbb{R}\right),\\
     0\leq z_{0}\left(x\right)\leq 1,  \quad \forall x\in\mathbb{R},\nonumber
  \end{eqnarray*}
  and $H_0:=\|(v_0(x)-V(0,x), u_0(x)-U(0,x), \theta_0(x)-\Theta(0,x),z_0(x))\|_{H^1(\mathbb{R})}$ together with $v_\pm, u_\pm, \theta_\pm$ are assumed to be independent of $\delta$ and $a$.
\end{itemize}
Then there exist positive constants $\delta_0$ and $a_0$, which depend only on $\underline{V}$, $\underline{\Theta}$,  and $H_0$, such that
\begin{equation}\label{Assumptions on strength and radiation constant}
0<\delta\leq \delta_0,\quad 0<a\leq a_0,
\end{equation}
the system \eqref{1.1}, \eqref{1.2}, \eqref{1.3}, \eqref{1.4}-\eqref{1.6} admits a unique global solution $\left(v(t,x), u(t,x), \theta(t,x), z(t,x)\right)$ which satisfies
\begin{eqnarray*}
C_{1}^{-1}\leq v(t,x)&\leq&C_{1},\nonumber\\
C_{2}^{-1}\leq\theta(t,x)&\leq&C_{2},\\
0\leq z(t,x)&\leq& 1\nonumber
\end{eqnarray*}
for all $\left(t,x\right)\in [0,\infty)\times\mathbb{R}$ and
\begin{eqnarray}
&&\sup\limits_{0\leq t<\infty}\left\|\left(v-V, u-U, \theta-\Theta, z\right)(t)\right\|_{H^{1}(\mathbb{R})}^{2}\nonumber\\
&&+\int_{0}^{\infty}\left(\left\|\partial_{x}\left(v-V\right)(\tau)\right\|^{2}_{L^2(\mathbb{R})}
+\big\|\big(\partial_{x}\left(u-U\right), \partial_{x}\left(\theta-\Theta\right), \partial_{x}z\big)(\tau)\big\|^{2}_{H^{1}\left(\mathbb{R}\right)}\right)d\tau\\
&\leq& C.\nonumber
\end{eqnarray}
Here $C_{1}$, $C_{2}$, and $C$ are some positive constants depending only on $\underline{V}$, $\underline{\Theta}$, and $H_0$.

Moreover, it holds that
\begin{equation}
 \lim_{t\to +\infty}\sup_{x \in \mathbb{R}}\Big\{\left|\left(v(t,x)-V^R(t,x),u(t,x)-U^R(t,x),s(t,x)-\overline{s},
 z(t,x)\right)\right|\Big\}=0.
\end{equation}
\end{theorem}

\begin{remark} Here are some remarks concerning about Theorem \ref{Th1.1}:
\begin{itemize}
\item Note that the result in \cite{Li-Wang-Yang-2020} focuses on the case when $\mu\equiv0$ and $\kappa(v,\theta)\equiv constant$. As pointed out before, the initial perturbation between the initial data and approximation solution in \cite{Li-Wang-Yang-2020} need to be sufficiently small. Besides, an additional stability condition should also be imposed on the state of the specific volume $v(t,x)$ and the temperature $\theta(t,x)$ at the far field (see (1.14) in \cite{Li-Wang-Yang-2020}). Compared with the result obtained in \cite{Li-Wang-Yang-2020}, the result in this paper is the first one concerning on the stability analysis of viscous wave pattern to system \eqref{1.1}, \eqref{1.2}, \eqref{1.3}, \eqref{1.4}, \eqref{1.5}, and \eqref{1.6} with constant viscosity under the large initial perturbation. Moreover, we do not need impose the above additional stability condition in our study. Furthermore, and our method in this paper can also be applied to Navier-Stokes equations when thermodynamic variables satisfy the equations of state for ideal polytropic gases ($\lambda=0$, $a=0$);

\item We emphasis that the result in \cite{Liao-NARWA-2020} can not include the case when $\mu\equiv C$. Besides, the methods to deduce the uniform lower and upper bounds on the specific volume $v(t,x)$ and the absolute temperature $\theta(t,x)$ in our paper are also different from that developed in \cite{Liao-NARWA-2020}.

\item The nonlinear stability results are called local stability or global stability depending on whether the $H^{1}-$ norm of the initial perturbation is small or not. For the ideal polytropic gas, \cite{Duan-Liu-Zhao-TAMS-2009} and \cite{Huang-Wang-Xiao-2016} proved the rarefaction wave for the system $\eqref{1.1}_{1}$-$\eqref{1.1}_{3}$ ($\lambda=0$, $a=0$) are stable with large initial perturbation with the condition that the adiabatic exponent $\gamma$ ($C_{v}=\frac{R}{\gamma-1}$) is closing enough to $1$. However, such a condition is not natural for the radiative and reactive gas in the physical setting since $C_{v}=\frac{3}{2}R$ in our case. Obviously, the stability result we obtained in this paper is a ``global one" and we do not need the smallness assumption on $\gamma-1$.

\end{itemize}
\end{remark}

As we can see in the analysis performed in \cite{Kawashima-Matsmura-Nishihara-PJA-1986, Matsmura-Nishihara-JJAM-1986, Matsmura-Nishihara-CMP-1992, Matsmura-Nishihara-QAM-2000, Nishihara-Yang-Zhao-SIAM-2004} and from the estimate \eqref{2.10} obtained in Lemma 2.5 of this paper that the fact that $\tilde{p}(v,s)$ is a convex function of $v$ and $s$ plays an essential role in deducing the nonlinear stability of rarefaction waves of the one-dimensional compressible Navier-Stokes type equations. We note, however, that, from \eqref{p7}, \eqref{p8}, and \eqref{p9}, it is not clear whether $\tilde{p}(v,s)$ is a convex function of $v$ and $s$ or not for the case when the radiation constant $a>0$. To overcome such a difficulty, our main observation is that if both the specific volume $v$ and the absolute temperature $\theta$ are bounded from the above and below by some positive constants independent of the radiation constant $a$, then one can choose $a$ sufficiently small such that $\tilde{p}(v,s)$ is a convex function of $v$ and $s$ in the regime for $v$ and $\theta$ under consideration. It is worth to pointing out that
in the proof of Theorem 1.1, the smallness assumption we imposed on the radiation constant $a$ is used only to ensure that $\tilde{p}(v,s)$ is convex with respect to $(v,s)$ in the regime for $v$ and $\theta$ under our consideration and we do not use such a smallness assumption elsewhere to control certain nonlinear terms involved. The main purpose of such an analysis is that once we can imposed some other assumptions to guarantee that $\tilde{p}(v,s)$ is convex with respect to $(v,s)$ in the regime for $v$ and $\theta$ under our consideration, then we can deduce that similar result holds accordingly.

Our next result show that, if in addition to use the smallness of $a$ to guarantee that $\tilde{p}(v,s)$ is convex with respect to $(v,s)$ in the regime for $v$ and $\theta$ under our consideration, we also use such an assumption to control certain nonlinear terms involved, then we can get a similar stability result but with less restrictions on range of the parameter $b$ and $\beta$, which includes the most physically interesting radiation case $b=3$ (cf. \cite{Jiang-ZHeng-ZAMP-2014}).

\begin{theorem}\label{Th1.2}
Under similar assumptions imposed on the initial data $(v_0(x), u_0(x), \theta_0(x), z_0(x))$ and the radiation constant $a$, similar stability result still holds when $b>2, \quad 0\leq\beta<b+3$.
\end{theorem}

In order to deduce the main results of this paper, the key points in our analysis are the following:
\begin{itemize}
\item The first is to deduce the uniform positive lower and upper bounds on the specific volume $v(t,x)$ and the absolute temperature $\theta(t,x)$;
\item The second is to show that the above bounds on the specific volume $v(t,x)$ and the absolute temperature $\theta(t,x)$ are independent of the radiation constant, since only in this case, we can choose $a>0$ sufficiently small such that $\tilde{p}(v,s)$ is a convex function of $v$ and $s$.
\end{itemize}

We are now in a position to state our main ideas to overcome the above difficulties, especially on the way to yield the uniform upper bound on the absolute temperature $\theta(t,x)$. To this end, we first recall that for the case when $a=0$ and $\kappa_2=0$ in \eqref{1.3}, \eqref{entropy}, and \eqref{1.4}, that is the equations of a viscous heat-conductive ideal polytropic gas with constant nondegenerate transport coefficients, the nonlinear stability of some basic wave patterns with large initial perturbation are obtained in \cite{Huang-Wang-IUMJ-2016, Wan-Wang-Zou-Nonlinearity-2016, Wan-Wang-Zhao-JDE-2016} for the whole range of the adiabatic exponent $\gamma>1$. The method used in \cite{Huang-Wang-IUMJ-2016, Wan-Wang-Zou-Nonlinearity-2016, Wan-Wang-Zhao-JDE-2016} to deduce the upper bound on the absolute temperature $\theta(t,x)$ is motivated by \cite{Li-Liang-2014}, which relies on the following Sobolev inequality
\begin{eqnarray*}
\left\|\theta(t)-1\right\|^{2}_{L^{\infty}(\mathbb{R})}\leq C\left\|\theta(t)-1\right\|_{L^2(\mathbb{R})}\left\|\theta_{x}(t)\right\|_{L^2(\mathbb{R})}
\leq C\left(1+\left\|\theta\right\|_{L^\infty([0,T]\times\mathbb{R})}\right).
\end{eqnarray*}
However, such a method loses its power for the case $\kappa_2\not=0$ since some nonlinear terms caused by the thermal conductivity $\kappa(v,\theta)=\kappa_1+\kappa_2v\theta^b$ can not be controlled properly when we deduce the estimate on $\left\|\theta_{x}(t)\right\|_{L^2(\mathbb{R})}$ by employing the argument developed in \cite{Li-Liang-2014}.

To overcome such a difficulty, for the case $(v_-,u_-,\theta_-)=(v_+,u_+,\theta_+):=(v_\infty,u_\infty,\theta_\infty)$, that is for the case when the far field of the initial data $(v_0(x), u_0(x),\theta_0(x))$ are equal, the argument developed in \cite{Liao-Zhao-JDE-2018} is to introduce the following auxiliary functions
\begin{eqnarray}\label{Auxiliary function for equal far fields}
\widetilde{X}(t):&=&\int_{0}^{t}\int_{\mathbb{R}}\left(1+\theta^{b+3}(s,x)\right)\theta_{t}^{2}(s,x)dxds,\nonumber\\ \widetilde{Y}(t):&=&\max\limits_{s\in(0,t)}\int_{\mathbb{R}}\left(1+\theta^{2b}(s,x)\right)\theta_{x}^{2}(s,x)dx,\\
\widetilde{Z}(t):&=&\max\limits_{s\in(0,t)}\int_{\mathbb{R}}u_{xx}^{2}(s,x)dx,\nonumber\\
\widetilde{W}(t):&=&\int_{0}^{t}\int_{\mathbb{R}}u^{2}_{xt}(s,x)dxds\nonumber
\end{eqnarray}
and then try to deduce certain estimates between them by employing the structure of the system \eqref{1.1}, \eqref{1.2}, \eqref{1.3}, and \eqref{1.4} under our consideration, from which one can deduce the desired upper bound on the absolute temperature $\theta(t,x)$. A key point in the analysis there is that the basic energy estimates based on the entropy $\widetilde{\eta}(v,u,\theta,v_\infty,u_\infty,\theta_\infty)$ normalized around the constant state $(v,u,\theta)=(v_-, u_-, \theta_-)$
\begin{eqnarray*}
\widetilde{\eta}(v,u,\theta,v_\infty,u_\infty,\theta_\infty)&=&
C_v\theta_\infty\Phi\left(\frac{\theta}{\theta_\infty}\right) +R\theta_\infty\Phi\left(\frac{v}{v_\infty}\right)+\frac 12(u-u_\infty)^2 +\frac{av}{3}(\theta-\theta_\infty)^2 \left(3\theta^2+2\theta_\infty\theta+\theta^2_\infty\right),\\
\Phi(x)&=&x-\ln x-1
\end{eqnarray*}
can yield a $L^4_{\textrm{loc}}(\mathbb{R})-$estimate on $\theta(t,x)$. From such an estimate, one can get by employing the argument developed in \cite{Kazhikov-Shelukhin-JAMM-1977} that, cf. (2.53) in \cite{Liao-Zhao-JDE-2018}
\begin{equation}\label{J1}
 \|\theta(t)\|_{L^{\infty}(\mathbb{R})}\lesssim 1 +\widetilde{Y}(t)^{\frac{1}{2b+6}}
\end{equation}
and the estimate \eqref{J1} plays an essential role in \cite{Liao-Zhao-JDE-2018} to deduce the upper bound of $\theta(t,x)$.

But for the case considered in this paper, $(v_-,u_-,\theta_-)\not= (v_+,u_+,\theta_+)$, since, as we pointed out before, we need to use the smallness of the radiation constant $a$ to ensure that $\tilde{p}(v,s)$ is a convex function of $v$ and $s$, although we can still construct a convex entropy $\eta(v,u,\theta;V,U,\Theta)$ normalized around the profile $(v,u,\theta)=(V(t,x), U(t,x),\Theta(t,x))$
\begin{eqnarray}\label{2.4}
\eta(v,u,\theta;V,U,\Theta)=C_{v}\Theta\Phi\left(\frac{\theta}{\Theta}\right) +R\Theta\Phi\left(\frac{v}{V}\right)+\frac{1}{2}(u-U)^2
+\frac{av(\theta-\Theta)^{2}}{3}\left(3\theta^{2}+2\theta\Theta+\Theta^{2}\right),
\end{eqnarray}
to yield a similar estimates, cf. \eqref{2.10} obtained in Lemma 2.5, to guarantee that the estimates we obtained on $\theta(t,x)$ does not depend on $a$, we can only use the boundedness of $\int_{\mathbb{R}}\Phi\left(\frac{\theta}{\Theta}\right)dx$. Moreover, the construction of the auxiliary functions $X(t), Y(t),$ and $Z(t)$ should also be modified accordingly as follows:
\begin{eqnarray}\label{4.1}
X(t):&&=\int_{0}^{t}\int_{\mathbb{R}} \left(1+\theta^{b}(s,x)\right)\chi_{t}^{2}(s,x)dxds,\nonumber\\
Y(t):&&=\sup\limits_{s\in(0,t)}\int_{\mathbb{R}} \left(1+\theta^{2b}(s,x)\right)\chi_{x}^{2}(s,x)\mathrm{d}x,\\
Z(t):&&=\sup\limits_{s\in(0,t)}\int_{\mathbb{R}} \psi_{xx}^{2}(s,x)\mathrm{d}x,\nonumber
\end{eqnarray}
where $\varphi(t,x):=v(t,x)-V(t,x), \psi(t,x)=u(t,x)-U(t,x), \chi(t,x)=\theta(t,x)-\Theta(t,x)$.

A consequence of the above modifications is that instead of the estimate \eqref{J1}, one has, cf. the estimate \eqref{4.2} in Lemma 4.1
\begin{equation}
\|\theta(t)\|_{L^{\infty}(\mathbb{R})}\lesssim 1 +Y(t)^{\frac{1}{2b+3}}.
\end{equation}
The above changes make it harder to deduce the upper bound of $\theta(t,x)$, especially to yield a nice bound on the term $I_{17}$ in \eqref{4.8} can not be controlled by exploiting the method used in \cite{Liao-Zhao-JDE-2018} to estimate the corresponding term, i.e. the term $I_{8}$ in \cite{Liao-Zhao-JDE-2018}.

Our strategy to overcome the above difficulties can be summarized as in the following:
\begin{itemize}
\item The smallness of the strength of the rarefaction waves is made full use of to control the nonlinear terms originated from the nonlinearities of equations, the interactions of rarefaction waves from different families and the interaction between the solutions and the rarefaction waves;

\item The specific volume $v(t,x)$ is shown to be uniformly bounded from below and above with respect to space and time variables through delicate analysis based on the basic energy estimate and the cut-off technique used by \cite{Jiang-CMP-1999,Liao-Zhao-JDE-2018}. It is worth to emphasizing that the positive lower and upper bounds we derived are independent of $\delta$ and $a$;

\item Motivated by \cite{Kawohl-JDE-1985, Liao-Zhao-JDE-2018}, we introduce the auxiliary functions $X(t)$, $Y(t)$, and $Z(t)$ defined by \eqref{4.1} to derive the desired upper bound of $\theta(t,x)$, especially to yield a nice estimate on the term $I_{17}$ given in \eqref{4.8}. To this end, we first derive bounds on $\left\|\varphi_{x}(t)\right\|_{L^2(\mathbb{R})}$ and $\displaystyle\int_{0}^{t}\int_{\mathbb{R}}\chi_{xx}^{2}dxd\tau$ in terms of $\left\|\theta\right\|_{L^\infty([0,T]\times\mathbb{R})}$ as in Lemma 3.5 and Lemma 3.6. Then by using Sobolev's inequality and Lemma 3.6, the term $\int_{0}^{t}\int_{\mathbb{R}}\left(1+\theta^{b}\right)\psi_{x}^{4}dxd\tau$ can be estimated as follows:
   \begin{eqnarray}\label{J2}
    \int_{0}^{t}\int_{\mathbb{R}}\left(1+\theta^{b}\right)\psi_{x}^{4}dxd\tau &\leq&
    C\left(1+\left\|\theta\right\|^{b}_{L^\infty([0,T]\times\mathbb{R})}\right)\int_{0}^{t} \left\|\psi_{x}\right\|^{2}_{L^\infty(\mathbb{R})}
     \left\|\psi_{x}\right\|^{2}_{L^2(\mathbb{R})}d\tau\nonumber\\
     &\leq&
    C\left(1+\left\|\theta\right\|^{b}_{L^\infty([0,T]\times\mathbb{R})}\right)\int_{0}^{t}\left\|\psi_{x}\right\|^{3}_{L^2(\mathbb{R})}
    \left\|\psi_{xx}\right\|_{L^2(\mathbb{R})}d\tau\\
    &\leq&
    C\left(1+\left\|\theta\right\|^{b+3}_{L^\infty([0,T]\times\mathbb{R})}\right)
    \left(\int_{0}^{t}\left\|\psi_{x}\right\|^{2}_{L^2(\mathbb{R})}d\tau\right)^{\frac{1}{2}}
    \left(\int_{0}^{t}\left\|\psi_{xx}\right\|^{2}_{L^2(\mathbb{R})}d\tau\right)^{\frac{1}{2}}\nonumber\\
     &\leq&
    C\left(1+\left\|\theta\right\|^{b+5}_{L^\infty([0,T]\times\mathbb{R})}\right).\nonumber
\end{eqnarray}
Note that we do not need to introduce the additional function $W(t)$  as in \cite{Liao-Zhao-JDE-2018}, cf. (2.51) and (2.70) in \cite{Liao-Zhao-JDE-2018}.

\end{itemize}

Finally, we point out that there are a lot of results concerning on the stability analysis of viscous wave pattern to the 1d compressible Navier-Stokes equations. We refer to \cite{He-Tang-Wang-Acta Math Sci-2016,Kawashima-Matsmura-CMP-1985,Liu-CPAM-1986,Tang-Zhang-Acta Math Sci-2018} for the viscous shock wave, \cite{Duan-Liu-Zhao-TAMS-2009,Huang-Wang-Xiao-2016,Liu-Xin-CMP-1988,Matsmura-Nishihara-JJAM-1986,Matsmura-Nishihara-CMP-1992,
Matsmura-Nishihara-QAM-2000,Nishihara-Yang-Zhao-SIAM-2004} for the rarefaction wave, \cite{Hong-JDE-2012,Huang-Mastsumura-Xin-ARMA-2006,Huang-Xin-Yang-Adavance-2008,Huang-Zhao-RSMUP-2003,Liu-Xin-AJM-1997} for the viscous contact wave, and \cite{Huang-Liao-M3AS-17,Huang-Li-Matsumura-ARMA-2010,Huang-Matsumura-CMP-2009,Huang-Wang-IUMJ-2016} for the superpositions of the above three wave patterns. For more references in this direction, please refer to \cite{Ducomet-Zlotnik-NonliAnal-2005,Kawashima-Nakamura-Nishibata-Zhu-M3AS-2010,Liao-Zhang-JMAA-2018,Qin-Wang-SIMA-2009,
Smoller-1999,Wan-Wang-Zou-Nonlinearity-2016, Wan-Wang-Zhao-JDE-2016, Wang-Zhao-M3AS-2016} and references therein.

The paper is organized as follows: we first give some basic energy estimates and some properties of the smooth approximation of the rarefaction wave solutions in Section 2. In Section 3, we derive the uniform-in-time lower and upper bounds of the specific volume $v(t,x)$ which are also independent of $\delta$ and $a$. Then the uniform-in-time, $\delta$ and $a$ independent upper bound of the absolute temperature $\theta(t,x)$ will be obtained in Section 4. Furthermore, a local-in-time lower bound on the absolute temperature will be deduced in Section 5. The proof of our main results are given in Section 6. Note that although the lower bound on the absolute temperature $\theta(t,x)$ obtained in Section 5 depends on time $t$, it is sufficient to prove the main theorem in this paper by combining these \emph{a priori} estimates with the continuation argument introduced in \cite{Liao-Zhao-JDE-2018}.

\bigbreak

\noindent\textbf{Notations.} In what follows, $C$ represents a generic positive constant,  which is independent of $t$, $\delta$, $a$, and $x$ but may depend on $v_\pm$, $u_\pm$, $\theta_\pm$, $\underline{V},$ $\overline{V},$ $\underline{\Theta},$ $\overline{\Theta}$, and $H_0$. Notice that the value of it may change from line to line. $C_i(\cdot,\cdot)(i \in \mathbb{Z}_+)$ stands for some generic constants depending only on the quantities listed in the parentheses and $\epsilon$ denotes some small positive constant.

For two quantities $B$ and $B'$, if there is a generic positive constant $C> 0$ independent of $t$, $\delta$, $a$, and $x$ such that $B \leq CB'$, we take the note $B\lesssim B'$, while $B \sim B'$ means that $B\lesssim B'$ and $B'\lesssim B$.
Moreover, for two functions $f(x)$ and $g(x)$, $f(x)\sim g(x)$ as $x \to x_{0}$ means that there exists a generic positive constant $C>0$ which is independent of $t$, $\delta$, $a$, and $x$ but may depend on $v_\pm$, $u_\pm$, $\theta_\pm$, $\underline{V},$ $\overline{V},$ $\underline{\Theta},$ $\overline{\Theta}$, and $H_0$ such that $C^{-1}f(x)\leq g(x)\leq Cf(x)$ in a neighborhood of $x_{0}$.  $H^l(\mathbb{R})(l\geq0)$ denotes the usual Sobolev space with standard norm $\|\cdot\|_l$, and for brevity, we take  $\|\cdot\|:=\|\cdot\|_0$ to denote the usual $L^2$-norm. For $1\leq p\leq +\infty, f(x) \in L^p(\mathbb{R})$, $\|f\|_{L^p}=(\int_{\mathbb{R}}|f(x)|^p \mathrm{d} x)^{\frac{1}{p}}$. It is easy to see that $\|f\|_{L^2}=\|\cdot\|$. Finally, $\|\cdot\|_{L^\infty}$ and $\|\cdot\|_{\infty}$ are used to denoted $\|\cdot\|_{L^\infty(\mathbb{R})}$ and $\|\cdot\|_{L^\infty([0,t]\times \mathbb{R})}$, respectively.

\section{Preliminaries}

First of all, (\ref{1.1}), (\ref{1.13}), and (\ref{1.15}) tell us that $(V(t,x), U(t,x), S(t,x),0)$ solves the following problem
\begin{eqnarray*}
  V_t-U_x&=&0,\\
  U_{t}+p(V,\Theta)_{x}&=&g(V,\Theta)_{x},\\
 \left(e(V,\Theta)+\frac{U^{2}}{2}\right)_{t} +\left(Up(V,\Theta)\right)_{x}&=&q(V,\Theta),\\
   \Theta_{t}+\frac{\Theta p_{\Theta}(V,\Theta)}{e_{\Theta}(V,\Theta)}U_{x} &=&r(V,\Theta),\\
   S_t&=&0,
\end{eqnarray*}
where
\begin{eqnarray*}
  g(V,\Theta)&=&p(V,\Theta)-p(V_{1},\Theta_{1}) -p(V_{3},\Theta_{3})-p(v_{m},\theta_{m}),\\
  q(V,\Theta)&=&\left(e(V,\Theta)-e(V_{1},\Theta_{1})-e(V_{3},\Theta_{3})\right)_{t} +\left(\frac{U^{2}}{2}-\frac{U^{2}_{1}}{2}-\frac{U^{2}_{3}}{2}\right)_{t}\\
  &&+\left(Up(V,\Theta)-U_{1}p(V_{1},\Theta_{1})-U_{3}p(V_{3},\Theta_{3})\right)_{x},\\
  r(V,\Theta)&=&\frac{\Theta p_{\Theta}(V,\Theta)}{e_{\Theta}(V,\Theta)}U_{x} -\frac{\Theta_{1}p_{\Theta}(V_{1},\Theta_{1})}{e_{\Theta}(V_{1},\Theta_{1})}U_{1x}-
   \frac{\Theta_{3}p_{\Theta}(V_{3},\Theta_{3})}{e_{\Theta}(V_{3},\Theta_{3})}U_{3x},\\[1mm]
\end{eqnarray*}
and $\theta_{m}=\widetilde{\theta}(v_{m},\overline{}\bar{s})$.

Due to the fact that $\omega_{0}(x)$ is strictly increasing, we can deduce the following lemma. (cf.\cite{Duan-Liu-Zhao-TAMS-2009,Liao-NARWA-2020})
\begin{lemma}
 For each $i \in \{1,3\},$ the Cauchy problem (\ref{1.14}) admits a unique global smooth solution $\omega_i(t,x)$ which satisfies the following properties:
\begin{itemize}
  \item[(i).] $\omega_- < \omega_i(t,x) < \omega_{+}$, \quad $\omega_{ix}(t,x)>0$ for each $(t,x) \in \mathbb{R}_+ \times \mathbb{R}$;
  \item[(ii).] For any $p$ with $1\leq p\leq\infty$, there exists a constant $C_{p.q}$, depending only on $p,q$, such that
\begin{eqnarray*}
 \left\|\omega_{ix}(t)\right\|_{L^P}^p &\leq& C_{p,q} \min\left\{\epsilon^{p-1}\tilde{\omega}_i^p,\tilde{\omega}_i t^{-p+1}\right\},\\
 \left\|\omega_{ixx}(t)\right\|_{L^P}^p &\leq& C_{p,q} \min\left\{\epsilon^{2p-1}\tilde{\omega}_i^p,
 \epsilon^{(p-1)(1-\frac{1}{2q})}\tilde{\omega}_i^{-\frac{p-1}{2q}}t^{-p-\frac{p-1}{2q}}\right\};
\end{eqnarray*}
  \item[(iii).] If $0< \omega_{i-}(<\omega_{i+})$ and $q$ is suitably large, then
   \begin{eqnarray*}
  \left|\omega_{i}(t,x)-\omega_{i-}\right|&\leq&  C \tilde{\omega}_i\left(1+(\epsilon x)^2\right)^{-\frac{q}{3}}\left(1+\left(\epsilon \omega_{i-}t\right)^2\right)^{-\frac{q}{3}},\quad  x\leq0,\\
\left|\omega_{ix}(t,x)\right|&\leq& C\epsilon \tilde{\omega}_i\left(1+(\epsilon x)^2\right)^{-\frac{1}{2}}\left(1+\left(\epsilon \omega_{i+}t\right)^2\right)^{-\frac{q}{2}},\quad  x\leq0;
\end{eqnarray*}
  \item[(iv).] If $(\omega_{i-})<\omega_{i+}\leq 0$ and $q$ is suitably large, then
   \begin{eqnarray*}
  \left|\omega_{i}(t,x)-\omega_{i+}\right|&\leq&  C \tilde{\omega}_i\left(1+\left(\epsilon x\right)^2\right)^{-\frac{q}{3}}\left(1+\left(\epsilon \omega_{i-}t\right)^2\right)^{-\frac{q}{3}}, \quad x\leq 0,\\
\left|\omega_{ix}(t,x)\right|&\leq& C\epsilon \tilde{\omega}_i\left(1+\left(\epsilon x\right)^2\right)^{-\frac{1}{2}}\left(1+\left(\epsilon \omega_{i+}t\right)^2\right)^{-\frac{q}{2}},\quad  x\leq0;
\end{eqnarray*}
   \item[(v).] $\lim\limits_{t \to +\infty}\sup\limits_{x \in {{R}}}\left|\omega_i(t,x)-\omega_{i}^{{{R}}} \left(\frac{x}{t}\right)\right|=0$.
 \end{itemize}
 Here $\tilde{\omega}_i=\omega_{i+}-\omega_{i-}>0$ and $\omega_i^{{{R}}}\left(\frac{x}{t}\right)$ is the unique rarefaction wave solution of the corresponding Riemann problem of $(1.14)_1$, i.e.,
       \begin{equation*}
    \omega_i^{{{R}}}(\xi)=\left\{
    \begin{array}{cc}
    \omega_{i-},\quad \xi\leq& \omega_{i-},\\[1mm]
    \xi,\quad\omega_{i-} \leq& \xi\leq \omega_{i+},\\[1mm]
    \omega_{i+},\quad \xi\geq &\omega_{i+}.
     \end{array}
      \right.
\end{equation*}

\end{lemma}

Owing to Lemma 2.1, (\ref{1.15}), and (\ref{1.16}), we can conclude that (cf.\cite{Duan-Liu-Zhao-TAMS-2009,Liao-NARWA-2020})
\begin{lemma} Letting $\epsilon=\delta$, $q=2$, the smooth approximations $(V(t,x), U(t,x),\Theta(t,x),0)$ constructed in (\ref{1.15}) and (\ref{1.16}) have the following properties:
\begin{description}
  \item[(i).] $V_t(t,x)=U_x(t,x)>0$ for each $(t,x)\in \mathbb{R}_{+}\times \mathbb{R}$;
  \item[(ii).] For any $p$ with $1\leq p\leq\infty$ there exists a constant $C_p$, depending only on $p$, such that
  \begin{eqnarray*}
 \left\|\left(V_x,U_x,\Theta_x\right)(t)\right\|_{L^p}^p &\leq& C_p\min\left\{\delta^{2p-1},\delta(t+1)^{-p+1}\right\},\\
\left\|\left(V_{xx},U_{xx},\Theta_{xx}\right)(t)\right\|_{L^p}^p &\leq& C_p\min\left\{\delta^{3p-1},\delta^{\frac{p-1}{2}} (t+1)^{-\frac{5p-1}{4}}\right\}.
\end{eqnarray*}
It is obvious that $\|V_x(t)\|^2_{L^2}$ is not integrable with respect to $t$. However we can get for any $r >0$ and $p>1$ that
 \begin{eqnarray*}
 \displaystyle\int^\infty_0\left\|\left(V_x,U_x,\Theta_x\right)(t)\right\|_{L^{2+r}}^{2+r}\mathrm{d}t &\leq& C(r)\delta,\\
 \displaystyle\int^\infty_0\left\|\left(V_{xx},U_{xx},\Theta_{xx}\right)(t)\right\|_{L^p}\mathrm{d}t  &\leq& C(p)\left(\frac{1}{\delta}\right)^{-\frac{1}{4}\left(1-\frac{1}{p}\right)};
\end{eqnarray*}
\end{description}
\begin{description}
  \item[(iii).] For each $p\geq 1$,
  $$\left\|\left(g(V,\Theta)_x, r(V,\Theta), q(V,\Theta)\right)(t)\right\|_{L^p}\leq C(p)\delta^{\frac{2}{3}}(t+1)^{-\frac{4}{3}}.$$
  Especially,
  $$\int^\infty_0\left\|\left(g(V,\Theta)_x, r(V,\Theta), q(V,\Theta)\right)(t)\right\|_{L^p} \mathrm{d}t  \leq C(p)\left(\frac{1}{\delta}\right)^{-\frac{1}{3}};$$
  \item[(iv).]$\lim\limits_{t \to +\infty}\sup\limits_{x \in \mathbb{R}}\left|\left(V(t,x),U(t,x),\Theta(t,x)\right)-\left(V^R\left(\frac{x}{t}\right),U^R\left(\frac{x}{t}\right),
      \Theta^R\left(\frac{x}{t}\right)\right)\right|=0;$
   \item[(v).] $\left|\left(V_t(t,x),U_t(t,x),\Theta_t(t,x)\right)\right|\leq O(1)\left|\left(V_x(t,x),U_x(t,x),\Theta_x(t,x)\right)\right|.$
\end{description}
\end{lemma}

Setting
\begin{eqnarray}\label{2.1}
&&(\varphi(t,x),\psi(t,x),\chi(t,x),\xi(t,x))\nonumber\\
&=&\left(v(t,x)-V(t,x),u(t,x)-U(t,x),\theta(t,x)-\Theta(t,x),s(t,x)-\bar{s}\right),
\end{eqnarray}
 we can deduce that $(\varphi(t,x), \psi(t,x), \chi(t,x),\xi(t,x), z(t,x))$ satisfies
\begin{eqnarray}\label{2.2}
  \varphi_{t}-\psi_{x}&=&0,\nonumber\\
  \psi_{t}+\left[p(v,\theta)-p(V,\Theta)\right]_{x}&=& \mu(\frac{u_{x}}{v})_{x}-g(V,\Theta)_{x},\nonumber\\
  \chi_{t}+\frac{\theta p_{\theta}(v,\theta)}{e_{\theta}(v,\theta)}\psi_{x}
  &=&\frac{1}{e_{\theta}(v,\theta)}\left(\frac{\mu u^{2}_{x}}{v}+\left(\frac{\kappa(v,\theta)\theta_{x}}{v}\right)_{x}+\lambda\phi z\right)\\
     &&-\left(\frac{\theta p_{\theta}(v,\theta)}{e_{\theta}(v,\theta)}
     -\frac{\Theta p_{\Theta}(V,\Theta)}{e_{\Theta}(V,\Theta)}\right)U_{x}-r(V,\Theta),\nonumber\\
  \xi_{t}&=&\frac{\mu u^{2}_{x}}{v\theta}+\left(\frac{\kappa(v,\theta)\theta_{x}}{v\theta}\right)_{x}+\frac{\kappa(v,\theta)\theta^{2}_{x}}{v\theta^{2}}
  +\frac{\lambda\phi z}{\theta},\nonumber\\
 z_{t}&=&\left(\frac{dz_{x}}{v^{2}}\right)_{x}-\phi z\nonumber
\end{eqnarray}
with initial data
\begin{eqnarray}\label{2.3}
&&\left(\varphi(t,x),\psi(t,x),\chi(t,x),\xi(t,x),z(t,x)\right)|_{t=0}\nonumber\\
&=&\left(\varphi_{0}(x),\psi_{0}(x),\chi_{0}(x),\xi_{0}(x),z_0(x)\right)\\
&:=&\left(v_{0}(x)-V(0,x),u_{0}(x)-U(0,x),\theta_{0}(x)-\Theta(0,x), s_{0}(x)-\bar{S},z_0(x)\right).\nonumber
\end{eqnarray}

On the other hand, it is easy to see that $\eta(v,u,\theta;V,U,\Theta)$ defined by \eqref{2.4} is a convex entropy to the system (\ref{1.1}) around the smooth rarefaction wave profile $\left(V(t,x), U(t,x),\Theta(t,x), 0\right)$ which solves
\begin{eqnarray}\label{2.5}
&&\eta_{t}\left(v,u,\theta,V,U,\Theta\right)+\left(\left(p(v,\theta)-p(V,\Theta)\right)\psi\right)_{x}
+\left(\frac{\mu\Theta\psi^{2}_{x}}{v\theta}+\frac{\kappa(v,\theta)\Theta\chi^{2}_{x}}{v\theta^{2}}\right)\nonumber\\
&&+\left(\tilde{p}(v,s)-\tilde{p}(V,\bar{s})-\tilde{p_{v}}(V,\bar{s})\varphi-\tilde{p_{s}}(V,\bar{s})\xi\right)U_{x}+\frac{\lambda\phi z\Theta}{\theta}\nonumber\\
&=&\left(\frac{\mu\psi\psi_{x}}{v}+\frac{\kappa(v,\theta)\chi\chi_{x}}{v\theta}\right)_{x}
+\left\{\frac{2\mu U_{x}\chi\psi_{x}}{v\theta}-\frac{\mu U_{x}\psi\varphi_{x}}{v^{2}}-\frac{\kappa(v,\theta)\Theta_{x}\chi\varphi_{x}}{v^{2}\theta}\right.\\
&&\left.+\frac{\kappa(v,\theta)\Theta_{x}\chi\chi_{x}}{v\theta^{2}}\right\}
+\left(\frac{\mu\psi U_{xx}}{v}+\frac{\kappa(v,\theta)\chi\Theta_{xx}}{v\theta}\right)+\left(\frac{\mu U_{x}^{2}\chi}{v\theta}-\frac{\mu U_{x}V_{x}\psi}{v^{2}}
-\frac{\kappa(v,\theta)V_{x}\Theta_{x}\chi}{v^2\theta}\right)\nonumber\\
&&-q(V,\Theta)-g(V,\Theta)_{x}\psi+g(V,\Theta)_{x}U-r(V,\Theta)\xi+\lambda\phi z+\frac{\kappa_{x}(v,\theta)\chi\Theta_{x}}{v\theta}.\nonumber
\end{eqnarray}

We first give the set of functions $X(0, T;M_1,M_2)$ for which we seek the solutions of the system \eqref{2.2}-\eqref{2.3} as follows:
\begin{eqnarray*}
   &&X(0, T;M_1,M_2)\\
&:=&\left\{\left(\varphi, \psi,\chi,z\right)(t,x)\ \left|
   \begin{array}{c}
   \left(\varphi,\psi,\chi\right)(t,x)\in C\left(0,T;H^{1}\left(\mathbb{R}\right)\right),\\
   \left(\psi_{x}, \chi_{x}, z_{x}\right)(t,x)\in L^{2}\left(0,T;H^{1}\left(\mathbb{R}\right)\right),\\
   \psi_{xx}\left(t,x\right)\in L^{2}\left(\mathbb{R}\right),\\
   \varphi_{x}\left(t,x\right)\in L^{2}\left(0,T; L^{2}\left(\mathbb{R}\right)\right),\\
   M^{-1}_1\leq V(t,x)+\varphi(t,x)\leq M_1,\ \forall (t,x)\in[0,T]\times\mathbb{R},\\
  M^{-1}_2\leq \Theta(t,x)+\chi(t,x)\leq M_2,\ \forall (t,x)\in[0,T]\times\mathbb{R},\\
   z(t,x)\in  C\left(0,T;H^{1}\left(\mathbb{R}\right)\cap L^1(\mathbb{R})\right),\\
  0\leq z(t,x)\leq 1.
   \end{array}
   \right.
   \right\}.
  \end{eqnarray*}
Here $0<T\leq +\infty$, $M_{1}$, and $M_{2}$ are some positive constants.

For the local solvability of the Cauchy problem \eqref{2.2} and \eqref{2.3} in the above set of functions, one has

\begin{lemma} [Local existence] Under the assumptions listed in Theorem 1.1, there exists a sufficiently small positive constant $t_1$, which depends only on $\left\|\left(\varphi_0,~\psi_0,~\chi_{0},~z_{0}\right)\right\|_1$, $\underline{V}$, $\overline{V}$, $\underline{\Theta}$, and $\overline{\Theta}$, such that the Cauchy problem (\ref{2.2})-(\ref{2.3}) admits a unique smooth solution $(\varphi(t,x), \psi(t,x), \chi(t,x), z(t,x))\in X\left(0,t_1;M'_{1},M'_{2}\right)$ which  satisfies
\begin{equation*}
     \begin{cases}
    0<\left(M'_{1}\right)^{-1}\leq\varphi(t,x)+V(t,x)\leq M'_{1},\\[1mm]
    0<\left(M'_{2}\right)^{-1}\leq\phi(t,x)+\Theta(t,x)\leq M'_{2},\\[1mm]
     0\leq z(t,x)\leq 1
 \end{cases}
\end{equation*}
for all $(t,x)\in[0,t_1]\times\mathbb{R}$ and $$
\sup_{t\in[0,t_1]}\left\{\left\|\left(\varphi,\psi,\chi,z\right)(t)\right\|_1\right\}
\leq2\left\|\left(\varphi_{0}, \psi_{0}, \chi_{0}, z_{0}
\right)\right\|_1.$$
\end{lemma}

Suppose that such a local solution $(\varphi(t,x), \psi(t,x), \chi(t,x), z(t,x))$ constructed in Lemma 2.3 has been extended to the time step $t=T>t_1$ and satisfies the \emph{a priori} assumption
\begin{equation}\label{2.6}
0<M_1^{-1}\leq v(t,x)\leq M_1,\quad 0<M_2^{-1}\leq \theta(t,x)\leq M_2
 \end{equation}
 for all $x\in\mathbb{R},0\leq t\leq T$ and some generic positive constants $M_1,\ M_2$ (without loss of generality, we assume in the rest of this manuscript that $M_1\geq1, M_2\geq1$), what we want to do next is to deduce some energy type estimates in terms of $\left\|\left(\varphi_{0}, \psi_{0}, \chi_{0}, z_{0}\right)\right\|_1$, $\underline{V}$, $\overline{V}$, $\underline{\Theta}$, and $\overline{\Theta}$, but are independent of $M_1$ and $M_2$. Throughout this paper, we assume $\delta$ (the strength of the rarefaction waves) and the radiation constant $a$ is small enough such that
 \begin{eqnarray}
\delta M_{1}^{100}M_{2}^{100+100b}&\ll& 1,\label{2.7-delta}\\
a M_{1}^{100}M_{2}^{100+100b}&\ll& 1.\label{2.7-a}
\end{eqnarray}

The following lemma guarantees that $\tilde{p}(v,s)$ is a convex function with respect to $(v,s)$. In fact, from \eqref{p7}, \eqref{p8}, \eqref{p9}, the \emph{a priori} assumption \eqref{2.6}, and the assumption \eqref{2.7-a} imposed on the radiation constant $a$, we can get that
\begin{lemma} Suppose that $(\varphi(t,x), \psi(t,x), \chi(t,x), z(t,x))\in X(0,T;M_1,M_2)$  is a solution to the Cauchy problem \eqref{2.2} and \eqref{2.3} defined on the strip $\Pi_T:=[0,T]\times\mathbb{R}$ and satisfying the \emph{a priori} assumption \eqref{2.6}, then $\tilde{p}(v,s)$ is convex with respect to $v$ and $s$ provided that $a>0$ is sufficiently small such that \eqref{2.7-a} holds. Consequently, we have
\begin{eqnarray*}
\tilde{p}(v,s)-\tilde{p}(V,\bar{s})-\tilde{p_{v}}(V,\bar{s})\varphi-\tilde{p_{s}}(V,\bar{s})\xi\geq 0.
\end{eqnarray*}
\end{lemma}
\begin{remark} To ensure that we can find sufficiently small positive constants $\delta_0, a_0$ such that the assumptions \eqref{2.7-delta} and \eqref{2.7-a} hold for all $0<\delta\leq \delta_0, 0<a\leq a_0$, a sufficient condition is to show that the positive lower and upper bounds on $v(t,x)$ and $\theta(t,x)$ depends only on $\left\|\left(\varphi_{0}, \psi_{0}, \chi_{0}, z_{0}\right)\right\|_1$, $\underline{V}$, $\overline{V}$, $\underline{\Theta}$, and $\overline{\Theta}$, but are independent of $M_1$, $M_2$, $\delta$ and $a$.
\end{remark}

Now we give the following lemma concerning on the basic energy estimates about the solution $(\varphi(t,x), \psi(t,x)$, $\chi(t,x), z(t,x))$, which will be frequently used later on.

\begin{lemma} [Basic energy estimates] In addition to the conditions stated in Lemma 2.4, we assume further that \eqref{2.7-delta} holds, then we have for all $0\leq t\leq T$ that
\begin{eqnarray}\label{2.8}
   \int_{\mathbb{R}}z(t,x)dx+\int_{0}^{t}\int_{\mathbb{R}}\phi(\tau,x) z(\tau,x)dxd\tau\lesssim 1,
\end{eqnarray}
\begin{eqnarray}\label{2.9}
 \int_{\mathbb{R}}z^{2}(t,x)dx+\int_{0}^{t}\int_{\mathbb{R}}\left(\frac{d}{v^{2}}z_{x}^{2}+\phi z^{2}\right)(\tau,x)dxd\tau
  \lesssim 1,
\end{eqnarray}
\begin{eqnarray}\label{2.10}
&&\int_{\mathbb{R}}\eta(t,x)dx +\int^{t}_{0}\int_{\mathbb{R}}\left(\frac{\mu\Theta\psi^{2}_{x}}{v\theta}
+\frac{\kappa(v,\theta)\Theta\chi^{2}_{x}}{v\theta^{2}}\right)(\tau,x)dxd\tau +\int^{t}_{0}\int_{\mathbb{R}}\left(\frac{\lambda\Theta\phi z}{\theta}\right)(\tau,x)dxd\tau\nonumber\\
&&+\int^{t}_{0}\int_{\mathbb{R}}\Big[\left(\tilde{p}(v,s) -\tilde{p}\left(V,\bar{s}\right)-\tilde{p_{v}} \left(V,\bar{s}\right)\varphi-\tilde{p_{s}}(V,\bar{s})\xi\right)U_{x}\Big](\tau,x)dxd\tau
\lesssim 1.
\end{eqnarray}
\end{lemma}

\begin{proof}
The estimates \eqref{2.8} and \eqref{2.9} follows directly from $\eqref{1.1}_{4}$ and integrations by parts. As for \eqref{2.10}, we have by integrating \eqref{2.5} with respect to $t$ and $x$ over $(0,t)\times\mathbb{R}$ that
\begin{eqnarray}\label{2.11}
&&\int_{\mathbb{R}}\eta(t,x)dx+\int^{t}_{0}\int_{\mathbb{R}}\left(\frac{\mu\Theta\psi^{2}_{x}}{v\theta}
+\frac{\kappa(v,\theta)\Theta\chi^{2}_{x}}{v\theta^{2}}\right)
+\int_{0}^{t}\int_{\mathbb{R}}\frac{\lambda\Theta\phi z}{\theta}\nonumber\\
&&+\int^{t}_{0}\int_{\mathbb{R}} \left(\tilde{p}(v,s)-\tilde{p}(V,\bar{s}) -\tilde{p_{v}}(V,\bar{s})\varphi -\tilde{p_{s}}(V,\bar{s})\xi\right)U_{x}\nonumber\\
&=&\int_{\mathbb{R}}\eta_{0}dx+\sum_{j=1}^{5}I_{j}.
\end{eqnarray}

By virtue of Lemma 2.2, the \emph{a priori} assumption \eqref{2.6}, \eqref{2.7-delta}, \eqref{2.8}, and Cauchy-Schwarz's inequality, $I_{j}$ $(j=1,2,3,4,5)$ can be bounded as follows:
\begin{eqnarray}\label{2.12}
I_{1}&=&\int^{t}_{0}\int_{\mathbb{R}}
\left(\frac{2\mu U_{x}\chi\psi_{x}}{v\theta}-\frac{\mu U_{x}\psi\varphi_{x}}{v^{2}}-\frac{\kappa(v,\theta)\Theta_{x}\chi\varphi_{x}}{v^{2}\theta}
+\frac{\kappa(v,\theta)\Theta_{x}\chi\chi_{x}}{v\theta^{2}}\right)\nonumber\\
&\leq& \delta^{\frac{1}{4}}\left(\int_{0}^{t}\left\|\frac{\varphi_{x}}{v}(\tau)\right\|^{2}d\tau+\int^{t}_{0}\int_{\mathbb{R}}\left(\frac{\mu\Theta\psi^{2}_{x}}{v\theta}
+\frac{\kappa(v,\theta)\Theta\chi^{2}_{x}}{v\theta^{2}}\right)\right)\nonumber\\
&&+\delta^{-\frac{1}{4}}C\bigg(\int_{0}^{t}\left\|U_{x}\right\|^{\frac{1}{2}}_{L^{\infty}}\left\|U_{x}\right\|^{\frac{3}{2}}_{L^{\infty}}
\int_{\mathbb{R}}\left(\frac{\psi}{v}\right)^{2}dxd\tau+\int^{t}_{0}\int_{\mathbb{R}}\left(\frac{\mu U^{2}_{x}\chi^{2}}{v\Theta\theta}
+\frac{\kappa(v,\theta)\chi^{2}\Theta^{2}_{x}}{v\theta^{2}\Theta}\right)\nonumber\\
&&+\int_{0}^{t}\left\|\Theta_{x}\right\|^{2}_{L^{\infty}}
\int_{\mathbb{R}}\left(\frac{\kappa(v,\theta)\chi}{v\theta}\right)^{2}dxd\tau\bigg)\\
&\leq& \delta^{\frac{1}{4}}\left(\int_{0}^{t}\left\|\frac{\varphi_{x}}{v}\right\|^{2}d\tau+\int^{t}_{0}\int_{\mathbb{R}}\left(\frac{\mu\Theta\psi^{2}_{x}}{v\theta}
+\frac{\kappa(v,\theta)\Theta\chi^{2}_{x}}{v\theta^{2}}\right)\right)\nonumber\\
&&+\delta^{\frac{3}{4}}C\left(M_{1}^{2}+M_{1}M_{2}+M_{1}^{2}M_{2}^{2}\left(1+M_{1}M_{2}^{b}\right)^{2}+M_{1}M_{2}^{2}\left(1+M_{1}M_{2}^{b}\right)\right)
\int_{0}^{t}\left(1+\tau\right)^{-\frac{3}{2}} \left\|\left(\psi,\chi\right)(\tau)\right\|^{2}d\tau,\nonumber
\end{eqnarray}
\begin{eqnarray}\label{2.13}
I_{2}&=&\int^{t}_{0}\int_{\mathbb{R}}\left(\frac{\mu U_{xx}\psi}{v}+\frac{\kappa(v,\theta)\Theta_{xx}\chi}{v\theta}\right)\nonumber\\
&\lesssim&\int_{0}^{t}\left(\left\|U_{xx}\right\|+\left\|\Theta_{xx}\right\|
+\left\|U_{xx}\right\|\left\|\frac{\psi}{v}\right\|^{2}+\left\|\Theta_{xx}\right\|
\left\|\frac{\kappa(v,\theta)\chi}{v\theta}\right\|^{2}\right)d\tau\nonumber\\
&\lesssim&\delta^{\frac{1}{8}}+\delta^{\frac{1}{4}}\left(M_{1}^{2}+M_{1}^{2}M_{2}^{2}\left(1+M_{1}^{2}M_{2}^{2b}\right)\right)
\int_{0}^{t}\left(1+\tau\right)^{-\frac{9}{8}}\left\|\left(\psi,\chi\right)(\tau)\right\|^{2}d\tau,
\end{eqnarray}
\begin{eqnarray}\label{2.14}
I_{3}&=&\int^{t}_{0}\int_{\mathbb{R}}
\left(\frac{\mu U_{x}^{2}\chi}{v\theta}-\frac{\mu U_{x}\psi V_{x}}{v^{2}}-\frac{\kappa(v,\theta)\Theta_{x}\chi V_{x}}{v^{2}\theta}
\right)\nonumber\\
&\lesssim&\int_{0}^{t}\int_{\mathbb{R}}\left(\left|U_{x}\right|^{\frac{5}{2}}+\left|V_{x}\right|^{\frac{5}{2}}
+\left|\frac{V_{x}\chi}{v^{2}}\right|^{\frac{5}{3}}
+\left|\frac{U_{x}\chi}{v\theta}\right|^{\frac{5}{3}}+\left|\frac{\kappa(v,\theta)\Theta_{x}\chi}{v^{2}\theta}\right|^{\frac{5}{3}}\right)\nonumber\\
&\lesssim&\int_{0}^{t}\left\|\left(U_{x},V_{x},\Theta_{x}\right)\right\|^{\frac{5}{2}}_{L^{\frac{5}{2}}}d\tau
+\int^{t}_{0}\int_{\mathbb{R}}\left(\left|V_{x}\right|^{\frac{3}{2}}\left|\frac{\psi}{v^{2}}\right|^{2}
+\left|U_{x}\right|^{\frac{3}{2}}\left|\frac{\chi}{v\theta}\right|^{2}
+\left|\Theta_{x}\right|^{\frac{3}{2}}\left|\frac{\kappa(v,\theta)\chi}{v^{2}\theta}\right|^{2}\right)\nonumber\\
&\lesssim&\delta+\delta^{\frac{1}{2}}\left(M_{1}^{4}+M_{1}^{2}M_{2}^{2}+M_{1}^{4}M_{2}^{2}\left(1+M_{1}^{2}M_{2}^{2b}\right)\right)
\int_{0}^{t}\left(1+\tau\right)^{-\frac{5}{4}}\left\|\left(\psi,\chi\right)(\tau)\right\|^{2}d\tau,
\end{eqnarray}

\begin{eqnarray}\label{2.15}
I_{4}&=&\int^{t}_{0}\int_{\mathbb{R}}\left(-q(V,\Theta)-g(V,\Theta)_{x}\psi+g(V,\Theta)_{x}U-r(V,\Theta)\xi\right)\nonumber\\
&\lesssim&\int_{0}^{t}\left(\left\|q\left(V,\Theta\right)\right\|_{L^{1}}+\left\|g\left(V,\Theta\right)_{x}\right\|_{L^{1}}
+\left\|g\left(V,\Theta\right)_{x}\right\|+\left\|r\left(V,\Theta\right)\right\|
+\left\|g\left(V,\Theta\right)_{x}\right\|\left\|\psi\right\|^{2}+\left\|r\left(V,\Theta\right)\right\|\left\|\xi\right\|^{2}\right)d\tau\nonumber\\
&\lesssim&\delta^{\frac{1}{3}}+\delta^{\frac{2}{3}}\int_{0}^{t}\left(1+\tau\right)^{-\frac{4}{3}}\left\|\left(\psi,\xi\right)(\tau)\right\|^{2}d\tau,
\end{eqnarray}
and
\begin{eqnarray}\label{2.16}
I_{5}&=&\int^{t}_{0}\int_{\mathbb{R}}\left(\lambda\phi z+\frac{\kappa_{x}(v,\theta)\chi\Theta_{x}} {v\theta}\right)\nonumber\\
&\lesssim& 1+\int_{0}^{t}\int_{\mathbb{R}}\left(\frac{\theta^{b}\left|\varphi_{x}\Theta_{x}\chi\right|}{v\theta}
+\frac{\theta^{b}\left|V_{x}\Theta_{x}\chi\right|}{v\theta}+\theta^{b-2}\left|\Theta_{x}\chi_{x}\chi\right|
+\theta^{b-2}\left|\Theta_{x}^{2}\chi\right|\right)\nonumber\\
&\lesssim& 1+\delta^{\frac{1}{4}}\int_{0}^{t}\int_{\mathbb{R}}\left(\left|\frac{\varphi_{x}}{v}\right|^{2}
+\frac{\kappa(v,\theta)\Theta\chi_{x}^{2}}{v\theta^{2}}\right)+\int_{0}^{t}\int_{\mathbb{R}}\left(\left|V_{x}\right|^{\frac{5}{2}}
+\left|\Theta_{x}\right|^{\frac{5}{2}}+\left|\frac{\theta^{b-1}\chi\Theta_{x}}{v}\right|^{\frac{5}{3}}
+\left|\theta^{b-2}\chi\Theta_{x}\right|^{\frac{5}{3}}\right)\nonumber\\
&&+\delta^{-\frac{1}{4}}\int_{0}^{t}\int_{\mathbb{R}}\left(\theta^{2b-2}\left|\Theta_{x}\chi\right|^{2}
+\frac{v\theta^{2b-2}\chi^{2}\left|\Theta_{x}\right|^{2}}{\Theta\kappa(v,\theta)}\right)\\
&\lesssim& 1+\delta^{\frac{1}{4}}\left(1+\int_{0}^{t}\left\|\frac{\varphi_{x}}{v}(\tau)\right\|^{2}d\tau
+\int^{t}_{0}\int_{\mathbb{R}}\frac{\kappa(v,\theta)\Theta\chi^{2}_{x}}{v\theta^{2}}\right)\nonumber\\
&&+\left(\delta^{\frac{3}{4}}M_{2}^{2b-2}+\delta^{\frac{1}{2}}M_{1}^{2}M_{2}^{2b-2}
+\delta^{\frac{3}{4}}M_{2}^{b-2}+\delta^{\frac{1}{2}}M_{2}^{2b-4}\right)\int_{0}^{t}\left[\left(1+\tau\right)^{-\frac{3}{2}}+\left(1+\tau\right)^{-\frac{5}{4}}\right]
\left\|\chi(\tau)\right\|^{2}d\tau.\nonumber
\end{eqnarray}

Combining \eqref{2.7-delta}, \eqref{2.11}-\eqref{2.16}, and making use of Gronwall's inequality, we can deduce that
\begin{eqnarray}\label{2.17}
&&\int_{\mathbb{R}}\eta(t,x)dx+\int^{t}_{0}\int_{\mathbb{R}}\left(\frac{\mu\Theta\psi^{2}_{x}}{v\theta}
+\frac{\kappa(v,\theta)\Theta\chi^{2}_{x}}{v\theta^{2}}\right)
+\int_{0}^{t}\int_{\mathbb{R}}\frac{\phi z}{\theta}\nonumber\\
&&+\int^{t}_{0}\int_{\mathbb{R}}\left(\tilde{p}(v,s) -\tilde{p}(V,\bar{s})-\tilde{p_{v}}(V,\bar{s})\varphi -\tilde{p_{s}}(V,\bar{s})\xi\right)U_{x}\\
&\lesssim& 1+\delta^{\frac{1}{4}}M_{1}M_{2}\int^{t}_{0} \int_{\mathbb{R}}\frac{\theta\varphi_{x}^{2}}{v^{3}}.\nonumber
\end{eqnarray}

Now we turn to estimate the term $\int^{t}_{0}\int_{\mathbb{R}}\frac{\theta\varphi_{x}^{2}}{v^{3}}$. For this purpose, we multiply $\eqref{2.2}_{2}$ by $\frac{\varphi_{x}}{v}$ to deduce that
\begin{eqnarray}\label{2.18}
 &&\left[\frac{\mu}{2}\left(\frac{\varphi_{x}}{v}\right)^2-\frac{\varphi_{x}\psi}{v}\right]_{t}
 +\frac{R\theta\varphi^{2}_{x}}{v^{3}}+\left(\frac{\psi\psi_{x}}{v}\right)_{x}\nonumber\\
 &=&\left[\frac{\psi_{x}^{2}}{v}+\frac{p_{\theta}(v,\theta)\varphi_{x}\chi_{x}}{v}\right]
 +\left\{\frac{\left(p_{v}(v,\theta)-p_{V}(V,\Theta)\right)V_{x}\varphi_{x}}{v}
 +\frac{\left(p_{\theta}(v,\theta)-p_{\Theta}(V,\Theta)\right)\Theta_{x}\varphi_{x}}{v}\right\}\\
 &&+\left[\frac{U_{x}\psi\varphi_{x}}{v^{2}}-\frac{V_{x}\psi\psi_{x}}{v^{2}}+\frac{\mu V_{x}\psi_{x}\varphi_{x}}{v^{3}}\right]
 +\left[\frac{\mu V_{x}U_{x}\varphi_{x}}{v^{3}}-\frac{\mu U_{xx}\varphi_{x}}{v^{2}}\right]+\frac{g(V,\Theta)_{x}\varphi_{x}}{v}.\nonumber
\end{eqnarray}
Then we integrate \eqref{2.18} over $(0,t)\times\mathbb{R}$ to derive
\begin{eqnarray}\label{2.19}
 &&\left\|\frac{\varphi_{x}}{v}\left(t\right)\right\|^{2}+\int^{t}_{0}\int_{\mathbb{R}}\frac{R\theta\varphi^{2}_{x}}{v^{3}}\nonumber\\
 &\lesssim&1+\left\|\psi(t)\right\|^{2}
 +\underbrace{\int_{0}^{t}\int_{\mathbb{R}}\left[\frac{\psi_{x}^{2}}{v}+\frac{p_{\theta}(v,\theta)\varphi_{x}\chi_{x}}{v}\right]}_{I_{6}}\nonumber\\
 &&+\underbrace{\int_{0}^{t}\int_{\mathbb{R}}\left\{\frac{\left(p_{v}(v,\theta)-p_{V}(V,\Theta)\right)V_{x}\varphi_{x}}{v}
 +\frac{\left(p_{\theta}(v,\theta)-p_{\Theta}(V,\Theta)\right)\Theta_{x}\varphi_{x}}{v}\right\}}_{I_{7}}\\
 &&+\underbrace{\int_{0}^{t}\int_{\mathbb{R}}\left[\frac{U_{x}\psi\varphi_{x}}{v^{2}}-\frac{v_{x}\psi\psi_{x}}{v^{2}}
 +\frac{\mu V_{x}\psi_{x}\varphi_{x}}{v^{3}}\right]}_{I_{8}}
 +\underbrace{\int_{0}^{t}\int_{\mathbb{R}}\left[\frac{\mu V_{x}U_{x}\varphi_{x}}{v^{3}}-\frac{\mu U_{xx}\varphi_{x}}{v^{2}}\right]}_{I_{9}}\nonumber\\
 &&+\underbrace{\int_{0}^{t}\int_{\mathbb{R}} \frac{g(V,\Theta)_{x}\varphi_{x}}{v}}_{I_{10}}.\nonumber
\end{eqnarray}

Now we turn to estimate $I_{j}$ ($j=6,\cdots, 10$) term by term. In fact, we have from Lemma 2.2, \emph{a priori} assumption \eqref{2.6} and Cauchy-Schwarz's inequality that
\begin{eqnarray}
 I_{6}&\leq&\frac{1}{10}\int^{t}_{0}\int_{\mathbb{R}}\frac{R\theta\varphi^{2}_{x}}{v^{3}}
 +C\left(\left\|\theta\right\|_{\infty}\int^{t}_{0}\int_{\mathbb{R}}\frac{\mu\Theta\psi^{2}_{x}}{v\theta}
 +\left\|\frac{\theta^{2}p^{2}_{\theta}\left(v,\theta\right)}{\kappa(v,\theta)p_{v}\left(v,\theta\right)}\right\|_{\infty}
 \int^{t}_{0}\int_{\mathbb{R}}\frac{\kappa(v,\theta)\Theta\chi^{2}_{x}}{v\theta^{2}}\right)\label{2.20}\\
 &\leq&\frac{1}{10}\int^{t}_{0}\int_{\mathbb{R}}\frac{R\theta\varphi^{2}_{x}}{v^{3}}
 +C\left(M_{2}\int^{t}_{0}\int_{\mathbb{R}}\frac{\mu\Theta\psi^{2}_{x}}{v\theta}
 +\left(M_{2}+M_{1}^{2}M_{2}^{7}\right)\int^{t}_{0} \int_{\mathbb{R}}\frac{\kappa(v,\theta)\Theta\chi^{2}_{x}} {v\theta^{2}}\right),\label{2.21}
\end{eqnarray}
\begin{eqnarray}\label{2.22}
 I_{7}&\leq&\frac{1}{10}\int^{t}_{0}\int_{\mathbb{R}}\frac{R\theta\varphi^{2}_{x}}{v^{3}}
 +C\int^{t}_{0}\int_{\mathbb{R}}\left(\frac{\left(p_{v}(v,\theta)-p_{V}(V,\Theta)\right)^{2}V^{2}_{x}}{v\left(-p_{v}(v,\theta)\right)}
 +\frac{\left(p_{\theta}(v,\theta)-p_{\Theta}(V,\Theta)\right)^{2}\Theta^{2}_{x}}{v\left(-p_{v}(v,\theta)\right)}\right)\nonumber\\
 &\leq&\frac{1}{10}\int^{t}_{0}\int_{\mathbb{R}}\frac{R\theta\varphi^{2}_{x}}{v^{3}}
 +C\int^{t}_{0}\int_{\mathbb{R}}\left(\left(\frac{\chi^{2}}{v^{3}\theta}+\frac{\varphi^{2}}{v^{3}\theta}
 +\frac{\varphi^{2}}{v\theta}\right)V^{2}_{x}+\left(\frac{\varphi^{2}}{vV^{2}\theta}+v\chi^{2}\theta^{3}+v\theta\Theta^{2}\chi^{2}
 +\frac{v\Theta^{4}\chi^{2}}{\theta}\right)\Theta^{2}_{x}\right)\nonumber\\
 &\leq&\frac{1}{10}\int^{t}_{0}\int_{\mathbb{R}}\frac{R\theta\varphi^{2}_{x}}{v^{3}}
 +\delta C\left(M_{1}^{3}M+M_{1}M_{2}^{3}\right)\int^{t}_{0}\left(1+\tau\right)^{-\frac{3}{2}}\left\|\left(\varphi,\chi\right)(\tau)\right\|^{2}d\tau,
\end{eqnarray}
\begin{eqnarray}\label{2.23}
 I_{8}&\leq&\frac{1}{10}\int^{t}_{0}\int_{\mathbb{R}}\frac{R\theta\varphi^{2}_{x}}{v^{3}}
 +C\int^{t}_{0}\int_{\mathbb{R}}\left(\frac{U_{x}^{2}\psi^{2}}{v^{3}\left(-p_{v}(v,\theta)\right)}
 +\frac{V_{x}^{2}\psi^{2}}{v^{3}\left(-p_{v}(v,\theta)\right)}+\frac{V_{x}^{2}\psi_{x}^{2}}{v^{5}\left(-p_{v}(v,\theta)\right)}\right)\nonumber\\
 &\leq&\frac{1}{10}\int^{t}_{0}\int_{\mathbb{R}}\frac{R\theta\varphi^{2}_{x}}{v^{3}}
 +\delta^{4}CM_{1}^{3}\int^{t}_{0}\int_{\mathbb{R}}\frac{\mu\Theta\psi^{2}_{x}}{v\theta}
 +\delta CM_{1}M_{2}\int^{t}_{0}\left(1+\tau\right)^{-\frac{3}{2}}\left\|\psi(\tau)\right\|^{2}d\tau,
\end{eqnarray}
\begin{eqnarray}\label{2.24}
 I_{9}&\leq&\frac{1}{10}\int^{t}_{0}\int_{\mathbb{R}}\frac{R\theta\varphi^{2}_{x}}{v^{3}}
 +C\int^{t}_{0}\int_{\mathbb{R}}\left(\frac{U_{x}^{2}V_{x}^{2}}{v^{5}\left(-p_{v}(v,\theta)\right)}
 +\frac{U_{xx}^{2}}{v^{3}\left(-p_{v}(v,\theta)\right)}\right)\nonumber\\
 &\leq&\frac{1}{10}\int^{t}_{0}\int_{\mathbb{R}}\frac{R\theta\varphi^{2}_{x}}{v^{3}}
 +C\left(M_{1}^{3}M_{2}\int^{t}_{0}\int_{\mathbb{R}}\left(\left|U_{x}\right|^{\frac{5}{2}}+\left|V_{x}\right|^{10}\right)
 +M_{1}M_{2}\int^{t}_{0}\int_{\mathbb{R}}U_{xx}^{2}\right)\nonumber\\
  &\leq&\frac{1}{10}\int^{t}_{0}\int_{\mathbb{R}}\frac{R\theta\varphi^{2}_{x}}{v^{3}}
 +C\left(M_{1}M_{2}\delta^{\frac{1}{2}}+M_{1}^{3}M_{2}\delta\right),
\end{eqnarray}
and
\begin{eqnarray}\label{2.25}
 I_{10}&\leq&\frac{1}{10}\int^{t}_{0}\int_{\mathbb{R}}\frac{R\theta\varphi^{2}_{x}}{v^{3}}
 +C\int^{t}_{0}\int_{\mathbb{R}}\frac{\left(g(V,\Theta)_{x}\right)^{2}}{v\left(-p_{v}(v,\theta)\right)}\nonumber\\
 &\leq&\frac{1}{10}\int^{t}_{0}\int_{\mathbb{R}}\frac{R\theta\varphi^{2}_{x}}{v^{3}}
 +CM_{1}M_{2}\delta^{\frac{4}{3}}.
\end{eqnarray}

Plugging \eqref{2.21}-\eqref{2.25} into \eqref{2.19} yields
\begin{eqnarray}\label{2.26}
 &&\left\|\left(\frac{\varphi_{x}}{v}\right)(t)\right\|^{2} +\int^{t}_{0}\int_{\mathbb{R}}\frac{R\theta\varphi^{2}_{x}}{v^{3}}\nonumber\\
 &\lesssim&1+\left\|\psi(t)\right\|^{2}+\left(M_{1}^{2}M_{2}^{7}+\delta^{4}M_{1}^{3}M_{2}\right)\int^{t}_{0}\int_{\mathbb{R}}\left(\frac{\mu\Theta\psi^{2}_{x}}{v\theta}
+\frac{\kappa(v,\theta)\Theta\chi^{2}_{x}}{v\theta^{2}}\right)\\
&&+\delta\left(M_{1}^{3}M_{2}+M_{1}M_{2}^{3}\right)\int_{0}^{t}\left(1+\tau\right)^{-\frac{3}{2}}
\left\|\left(\varphi,\psi,\chi\right)(\tau)\right\|^{2}d\tau.\nonumber
\end{eqnarray}

Having obtained \eqref{2.17}, and \eqref{2.26}, we can deduce \eqref{2.10} immediately by the assumption \eqref{2.7-delta} and Gronwall's inequality.
\end{proof}

By repeating the argument developed in \cite{Chen-SIMA-1992}, we can deduce the pointwise bounds of $z(t,x)$. Here we omit the proof for brevity.

\begin{lemma} Under the conditions listed in Lemma 2.5, we have for all $(t,x)\in [0,T]\times\mathbb{R}$ that
  \begin{equation}\label{2.27}
  0\leq z(t,x)\leq 1.
  \end{equation}
\end{lemma}

\section{Uniform bounds for the specific volume}
The main purpose of this section is to deduce the uniform-in-time pointwise bounds for the specific volume $v\left(t,x\right)$ to the Cauchy problem \eqref{2.2}, \eqref{2.3}, which do not depend on $\delta$ and $a$. To this end, we first give the following lemma, which is a consequence of \eqref{2.10} and Jensen's inequality.

\begin{lemma}  Under the conditions listed in Lemma 2.5, we have for all $k\in\mathbb{Z}$ and $t\in[0,T]$ that there exist $a_{k}(t), b_{k}(t)\in \Omega_{k}:=[-k-1,k+1]$ such that
 \begin{eqnarray}\label{3.1}
 \int_{\Omega_{k}}v(t,x)\mathrm{d}x\sim 1,\quad \int_{\Omega_{k}}\theta(t,x)\mathrm{d}x\sim 1,\quad
v(t,a_{k}(t))\sim 1,\quad \theta(t,b_{k}(t))\sim 1.
  \end{eqnarray}
\end{lemma}

The next lemma is concerned with a rough estimate on $\theta(t,x)$ in terms of the entropy dissipation rate functional
$V(t)=\displaystyle\int_{\mathbb{R}}\left(\frac{\mu\Theta\psi^{2}_{x}}{v\theta}
+\frac{\kappa(v,\theta)\Theta\chi^{2}_{x}}{v\theta^{2}}\right)(t,x) dx$.

\begin{lemma} Under the conditions listed in Lemma 2.5, we have for $0\leq m\leq\frac{b+1}{2}$ and each $x\in\mathbb{R}$ (without loss of generality, we can assume that $x\in\Omega_k$ for some $k\in\mathbb{Z}$) that
\begin{eqnarray}\label{3.2}
\left |\theta^{m}\left(t,x\right)-\theta^{m}\left(t, b_{k}\left(t\right)\right)\right|\lesssim V^{\frac{1}{2}}\left(t\right)+1
\end{eqnarray}
holds for $0\leq t\leq T$ and consequently
\begin{eqnarray}\label{3.3}
 \left|\theta\left(t,x\right)\right|^{2m}\lesssim 1+V\left(t\right), \quad x\in\overline{\Omega}_{k},\ \  0\leq t\leq T.
\end{eqnarray}
\end{lemma}
\begin{proof} We have from \eqref{1.4} that
\begin{eqnarray}\label{3.4}
 &&\left|\theta^{m}\left(t,x\right)-\theta^{m}\left(t, b_{k}\left(t\right)\right)\right|\nonumber\\
 &\lesssim&\int_{\Omega_{k}}\left|\theta^{m-1}
 \left(\Theta_{x}+\chi_{x}\right)\right|dx\\
 &\lesssim&\left(\int_{\Omega_{k}}\frac{v\theta^{2m}}{1+v\theta^{b}}dx\right)^{\frac{1}{2}}\left(\int_{\Omega_{k}}\left(\frac{\mu\Theta\psi^{2}_{x}}{v\theta}
 +\frac{\kappa(v,\theta)\Theta\chi^{2}_{x}} {v\theta^{2}}\right)dx\right)^{\frac{1}{2}} +\int_{\Omega_{k}}\theta^{m-1}\left|\Theta_{x}\right|dx\nonumber\\
 &\lesssim& V^{\frac{1}{2}}(t)+M_{2}^{\frac{b-1}{2}}\delta^{2}\nonumber\\
 &\lesssim& V^{\frac{1}{2}}(t)+1.\nonumber
\end{eqnarray}
It is worth to point out that we have used the assumption $0\leq m\leq\frac{b+1}{2}$, boundedness of $\Omega_{k}$, \eqref{2.6}, \eqref{2.7-delta}, and \eqref{3.1} in deriving the above inequality.
\end{proof}

The next lemma will give a local representation of $v(t,x)$ by using the following cut-off function $\varphi(x)\in W^{1,\infty}\left(\mathbb{R}\right)$
\begin{equation}\label{3.5}
 \varphi\left(x\right)=
 \begin{cases}
 1, & $$x\leq k+1$$,\\
 k+2-x, & $$ k+1\leq x\leq k+2$$,\\
 0, & $$ x\geq k+2$$.
 \end{cases}
\end{equation}

\begin{lemma} Under the assumptions stated Theorem 1.1, we have for each $0\leq t\leq T$ that
\begin{eqnarray}\label{3.6}
v\left(t,x\right)=B\left(t,x\right)Q\left(t\right)+\frac{1}{\mu}\int_{0}^{t}\frac{B\left(t,x\right)Q\left(t\right)v\left(\tau,x\right)
p\left(\tau,x\right)}{B\left(\tau,x\right)Q\left(\tau\right)}d\tau,
\quad x\in\overline{\Omega}_k.
\end{eqnarray}
Here
\begin{eqnarray}\label{3.7}
B\left(t,x\right)&&:=v_{0}\left(x\right)\exp\left\{\frac{1}{\mu}\int_{x}^{\infty}\left(u_{0}\left(y\right)-u\left(t,y\right)\right)\varphi\left(y\right)dy\right\},\nonumber\\
Q\left(t\right)&&:=\exp\left\{\frac{1}{\mu}\int_{0}^{t}\int_{k+1}^{k+2}\sigma\left(\tau,y\right)\right\},\\
\sigma &&:=-p(v,\theta)+\frac{\mu u_{x}}{v}.\nonumber
\end{eqnarray}
\end{lemma}

With the above presentation in hand, we can deduce uniform-in-time pointwise bounds of $v\left(t,x\right)$ by repeating the argument used in \cite{Liao-Zhao-JDE-2018}, and we omit the proof for brevity.
\begin{lemma} Assume that the conditions listed in Lemma 2.5 hold, then there exists a positive constant $C_{1}$ which depends only on $\underline{V}$, $\overline{V}$, $\underline{\Theta}$, $\overline{\Theta}$, and $H_0$, but independent of $\delta$ and $a$, such that
	\begin{equation}\label{3.8}
	C_{1}^{-1}\leq v(t,x)\leq C_{1}, \quad\forall\ (t,x)\in [0,T]\times\mathbb{R}.
	\end{equation}
\end{lemma}

The following lemma is concerning on the  estimate on the term $\left\|\varphi_{x}\left(t\right)\right\|^{2}$, which will be frequently used later on.
\begin{lemma} Under the assumptions listed in Lemma 2.5, we have for any $0\leq t\leq T$ that
\begin{eqnarray}\label{3.9}
\left\|\varphi_{x}(t)\right\|^2 +\int_{0}^{t}\int_{\mathbb{R}}\theta(\tau,x)\varphi_{x}^{2}(\tau,x)dxd\tau\lesssim 1 +\|\theta\|_{\infty}.
\end{eqnarray}
\end{lemma}

\begin{proof}
In light of \eqref{2.10}, \eqref{2.20}, and \eqref{3.8}, we can conclude that
\begin{eqnarray}\label{3.10}
 I_{6}&\leq&\frac{1}{10}\int^{t}_{0}\int_{\mathbb{R}}\frac{R\theta\varphi^{2}_{x}}{v^{3}}
 +C\left(\left\|\theta\right\|_{\infty}+
 \left\|\frac{\theta^{2}p^{2}_{\theta}\left(v,\theta\right)}{\kappa(v,\theta)p_{v}\left(v,\theta\right)}\right\|_{\infty}\right)\nonumber\\
&\leq&\frac{1}{10}\int^{t}_{0}\int_{\mathbb{R}}\frac{R\theta\varphi^{2}_{x}}{v^{3}}
 +C\left(1+\left\|\theta\right\|_{\infty} +\|\theta\|^{(7-b)_{+}}_{\infty}\right).
\end{eqnarray}
Here $(7-b)_{+}:=\max\{0,7-b\}$.

Then inserting \eqref{3.10}, \eqref{2.22}-\eqref{2.25} into \eqref{2.19}, we can get \eqref{3.9} by employing \eqref{2.7-delta} and the assumption $b>6$.
\end{proof}

The next lemma pays attention to the estimate on the term $\int_{0}^{t}\left\|\psi_{xx}(\tau)\right\|^{2}d\tau$, which will be useful in deducing the upper bound of $\theta\left(t,x\right)$.

\begin{lemma} Under the conditions listed in Lemma 2.5, we have for any $0\leq t\leq T$ that
\begin{eqnarray}\label{a5.1}
\|\psi_{x}(t)\|^{2}+\int_{0}^{t}\left\|\psi_{xx}(\tau)\right\|^2d\tau\lesssim 1+\|\theta\|^{3}_{\infty}.
\end{eqnarray}
\end{lemma}

\begin{proof}
We multiply $\eqref{2.2}_{2}$ by $-\psi_{xx}$ to get
\begin{eqnarray}\label{a5.2}
&&\partial_{t}\left(\frac{\psi_{x}^{2}}{2}\right)+\frac{\mu\psi^{2}_{xx}}{v}-\left(\psi_{t}\psi_{x}\right)_{x}\nonumber\\
&=&
\left(p(v,\theta)-p(V,\Theta)\right)_{x}\psi_{xx}-\frac{\mu\psi_{xx}U_{xx}}{v}+g\left(V,\Theta\right)_{x}\psi_{xx}\\
&&+\frac{\mu\left(\psi_{x}\varphi_{x}\psi_{xx}+\psi_{x}V_{x}\psi_{xx} +U_{x}\varphi_{x}\psi_{xx}+V_{x}U_{x}\psi_{xx}\right)}{v^{2}}.\nonumber
\end{eqnarray}

Integrating \eqref{a5.2} with respect to $t$ and $x$ over $(0,t)\times\mathbb{R}$, we utilize Cauchy's inequality, Sobolev's inequality, Lemma 2.2, \eqref{3.8}, and \eqref{3.9} to find that
\begin{eqnarray}\label{a5.3}
&&\int_{\mathbb{R}}\frac{\psi_{x}^{2}}{2}dx+\int_{0}^{t}\int_{\mathbb{R}}\frac{\mu\psi_{xx}^{2}}{v}\nonumber\\
&\leq&\epsilon \int_{0}^{t}\int_{\mathbb{R}}\frac{\mu\psi_{xx}^{2}}{v}
+C\left(\epsilon\right)\int_{0}^{t}\int_{\mathbb{R}}\bigg[\left(1+a^{2}\theta^6\right)|\chi_{x}|^2
+\left|\left(V_{x},\Theta_{x}\right)\right|^2|\varphi|^2
+|\theta\varphi_{x}|^2+\chi^{2}|V_{x}|^{2}\nonumber\\
&&+a^{2}\left(1+\theta^4\right)|\Theta_{x}|^2|\chi|^2
+U_{xx}^{2}+\left|g\left(V,\Theta\right)_{x}\right|^{2}+\psi_{x}^{2}\varphi_{x}^{2}+\psi_{x}^{2}V_{x}^{2}+U_{x}^{2}\varphi_{x}^{2}
+V_{x}^{2}U_{x}^{2}\bigg]\\
&\leq&\epsilon \int_{0}^{t}\int_{\mathbb{R}}\frac{\mu\psi_{xx}^{2}}{v}
+C\left(\epsilon\right)\left(1+\|\theta\|^{(8-b)_{+}}_{\infty}+\|\theta\|^{2}_{\infty}+
\int_{0}^{t}\left\|\psi_{x}\right\|\left\|\psi_{xx}\right\|\left\|\varphi_{x}\right\|^{2}d\tau\right)\nonumber\\
&\leq&2\epsilon \int_{0}^{t}\int_{\mathbb{R}}\frac{\mu\psi_{xx}^{2}}{v}
+C\left(\epsilon\right)\left(1+\|\theta\|^{(8-b)_{+}}_{\infty}+\|\theta\|^{2}_{\infty}
+\left(1+\|\theta\|^{2}_{\infty}\right)\int_{0}^{t}\int_{\mathbb{R}}\frac{\psi^{2}_{x}}{\theta}\cdot\theta\right)\nonumber\\
&\leq&2\epsilon \int_{0}^{t}\int_{\mathbb{R}}\frac{\mu\psi_{xx}^{2}}{v}
+C\left(\epsilon\right)\left(1+\|\theta\|^{(8-b)_{+}}_{\infty} +\|\theta\|^{3}_{\infty}\right).\nonumber
\end{eqnarray}
If we Choose $\epsilon>0$ small enough and by using the assumption $b>6$, we can obtain \eqref{a5.1} immediately.
\end{proof}

\section{Uniform upper bound of the absolute temperature}

Now we are in a position to derive an estimate on the upper bound of $\theta\left(t,x\right)$. To this end, recall the definitions of the auxiliary functions $X(t), Y(t)$, and $Z(t)$ defined by \eqref{4.1}, we then try to deduce certain estimates among them by employing the special structure of system \eqref{2.2}.

Our first result is to show that $\|\theta(t)\|_{L^\infty}$, $\|\psi(t)\|$, and $\|\psi_x(t)\|_{L^\infty}$ can be controlled by $Y(t)$ and $Z(t)$, respectively.
\begin{lemma}Under the conditions listed in Lemma 2.5, we have for all $0\leq t\leq T$ that
\begin{eqnarray}
\label{4.2} \|\theta(t)\|_{L^{\infty}} &&\lesssim 1 +Y(t)^{\frac{1}{2b+3}},\\
\label{4.3} \sup_{\tau\in(0,t)}\| \psi_{x}(\tau)\|^2 &&\lesssim 1 +Z(t)^{\frac{1}{2}},\quad
\|\psi_{x}(t)\|_{L^{\infty}}\lesssim 1 +Z(t)^{\frac{3}{8}}.
\end{eqnarray}
\end{lemma}
\begin{proof} We assume that $x\in[-k-1,k+1]$ for some $k\in\mathbb{Z}$ and $x\geq b_{k}(t)$ and observe that
\begin{eqnarray*}
&&\left(\theta(t,x)-\Theta(t,x)\right)^{2b+3}\\
&=&(\theta(t,b_{k}(t))-\Theta(t,b_{k}(t)))^{2b+3}
+\int_{b_{k}(t)}^{x}\left(2b+3\right)\left(\theta(t,y)-\Theta(t,y)\right)^{2b+2}\chi_{x}(t,y)\mathrm{d}y\nonumber\\
&\lesssim& 1 +\|\left(\theta-\Theta\right)(t)\|_{L^{\infty}}^\frac{{2b+3}}{2}
\left[\int_{-k-1}^{k+1}\left(1+\Phi\left(\frac{\theta}{\Theta}\right)
\right) \mathrm{d}x\right]^{\frac{1}{2}} \left[\int_{\mathbb{R}} \left(\theta-\Theta\right)^{2b}\chi_{x}^{2}\mathrm{d}x\right]^{\frac{1}{2}}\nonumber \\
&\lesssim& 1 +\|\left(\theta-\Theta\right)(t)\|^\frac{{2b+3}}{2}_{L^{\infty}}Y^{\frac{1}{2}}(t).\nonumber
\end{eqnarray*}
Then applying Cauchy's inequality, we can obtain \eqref{4.2}.

Estimates \eqref{4.3} is a consequence of Gagliardo--Nirenberg and Sobolev inequalities. This completes the proof of Lemma 4.1.

\end{proof}

Our next result shows that $X(t)$ and $Y(t)$ can be bounded by $Z(t)$.
\begin{lemma}Under the conditions listed in Lemma 2.5, we have for $0\leq t\leq T$ that
\begin{eqnarray}\label{4.4}
X(t)+Y(t)\lesssim 1+Z(t)^{\frac{6b+9}{12b+4}}.
\end{eqnarray}
\end{lemma}

\begin{proof}
In the same manner as \cite{Kawohl-JDE-1985,Liao-Zhao-JDE-2018}, we set
\begin{eqnarray}\label{4.6}
 \mathbb{K}(v,\theta)=\int_{0}^{\theta}\frac{\kappa(v,\xi)}{v}d\xi=\frac{\kappa_{1}\theta}{v}+\frac{\kappa_{2}\theta^{b+1}}{b+1}.
  \end{eqnarray}
Then we can deduce that
  \begin{eqnarray}\label{4.7}
  \mathbb{K}_{t}(v,\theta)&=&\mathbb{K}_{v}(v,\theta)\psi_{x} +\mathbb{K}_{\theta}(v,\theta)\chi_{t}+\mathbb{K}_{v}(v,\theta)U_{x}+\mathbb{K}_{\theta}(v,\theta)\Theta_{t},\nonumber\\[1mm]
  \mathbb{K}_{x}(v,\theta)&=&\mathbb{K}_{v}(v,\theta)\varphi_{x} +\mathbb{K}_{\theta}(v,\theta)\chi_{x}+\mathbb{K}_{v}(v,\theta)V_{x}+\mathbb{K}_{\theta}(v,\theta)\Theta_{x},\nonumber\\[1mm]
  \mathbb{K}_{xt}(v,\theta)&=&\left(\mathbb{K}_{\theta}(v,\theta)\chi_{x}\right)_{t} +\left[\mathbb{K}_{vv}(v,\theta)(\psi_{x}+U_{x})+\mathbb{K}_{v\theta}(v,\theta)\left(\chi_{t}+\Theta_{t}\right)\right]\varphi_{x}
  +\mathbb{K}_{v}(v,\theta)\psi_{xx}\\
  &&+\left[\mathbb{K}_{vv}(v,\theta)(\psi_{x}+U_{x})+\mathbb{K}_{v\theta}(v,\theta)\left(\chi_{t}+\Theta_{t}\right)\right]V_{x}
  +\mathbb{K}_{v}(v,\theta)U_{xx}\nonumber\\
  &&+\left[\mathbb{K}_{v\theta}(v,\theta)(\psi_{x}+U_{x})+\mathbb{K}_{\theta\theta}(v,\theta)(\chi_{t}+\Theta_{t})\right]\Theta_{x}
  +\mathbb{K}_{\theta}(v,\theta)\Theta_{xt},\nonumber\\
    \left|\mathbb{K}_{v}(v,\theta)\right|&+&\left|\mathbb{K}_{vv} (v,\theta)\right|\lesssim\theta,\quad
    \left|\mathbb{K}_{\theta}(v,\theta)\right|\lesssim 1+\theta^{b}, \quad \left|\mathbb{K}_{v\theta}(v,\theta)\right|\lesssim 1,
    \quad \left|\mathbb{K}_{\theta\theta}(v,\theta)\right|\lesssim \theta^{b-1}.\nonumber
\end{eqnarray}
Hereafter, for simplicity of presentation, we use $\mathbb{K}, p, e, P,$ and $E$ to denote the terms $\mathbb{K}(v,\theta), p(v,\theta), e(v,\theta), p(V,\Theta),$ and $e(V,\Theta)$, respectively.

We multiply $\eqref{2.2}_{3}$ by $\mathbb{K}_{t}$ and integrate the result identity with respect to $t$ and $x$ over $(0,t)\times\mathbb{R}$ to find that
\begin{eqnarray}\label{4.8}
    &&\int_{0}^{t}\int_{\mathbb{R}}e_{\theta}\mathbb{K}_{\theta}\chi_{t}^2
     +\int_{0}^{t}\int_{\mathbb{R}}\left(\mathbb{K}_{\theta}\chi_{x}\right)\left(\mathbb{K}_{\theta}\chi_{x}\right)_{t}
    -\int_{0}^{t}\int_{\mathbb{R}}\mathbb{K}_{\theta}\mathbb{K}_{v\theta}U_{x}\chi_{x}^2\nonumber\\
    &=&\underbrace{\int_{0}^{t}\int_{\mathbb{R}}\left(\mathbb{K}_{\theta}\mathbb{K}_{\theta\theta}\chi_{t}\Theta_{x}^2+\mathbb{K}_{\theta}^2\chi_{t}\Theta_{xx}
     +\mathbb{K}_{\theta}\mathbb{K}_{v\theta} \chi_{t}\Theta_{x}V_{x}\right)}_{I_{11}}\nonumber\\
    &&+\underbrace{\int_{0}^{t}\int_{\mathbb{R}}\left\{\mathbb{K}_{\theta}\mathbb{K}_{\theta\theta}\chi_{x}^2\Theta_{t}
      +\mathbb{K}_{\theta}\mathbb{K}_{v\theta}\chi_{t}\varphi_{x}\Theta_{x}-\mathbb{K}_{\theta}\mathbb{K}_{v\theta}\chi_{x}\chi_{t}V_{x}
     -\left(\theta p_{\theta}-\frac{e_{\theta}\Theta P_{\Theta}}{E_{\Theta}}\right)U_{x}\mathbb{K}_{\theta} \chi_{t}\right\}}_{I_{12}}\nonumber\\
     &&+\underbrace{\int_{0}^{t}\int_{\mathbb{R}}\mathbb{K}_{\theta} \mathbb{K}_{v\theta}\chi_{x}^2\psi_{x}}_{I_{13}}
      -\underbrace{\int_{0}^{t}\int_{\mathbb{R}}\mathbb{K}_{\theta} \mathbb{K}_{v\theta}\chi_{t}\varphi_{x}\chi_{x}}_{I_{14}}\nonumber\\
     &&+\underbrace{\int_{0}^{t}\int_{\mathbb{R}}\left(\lambda\phi z-e_{\theta}r(v,\theta)\right)\mathbb{K}_{\theta}\chi_{t}}_{I_{15}}
     -\underbrace{\int_{0}^{t}\int_{\mathbb{R}}\theta p_{\theta}\mathbb{K}_{\theta}\psi_{x}\chi_{t}}_{I_{16}}
     +\underbrace{\int_{0}^{t}\int_{\mathbb{R}}\frac{\mu u_x^2\mathbb{K}_{\theta}\chi_{t}}{v}}_{I_{17}}.
\end{eqnarray}
Firstly, we find that
\begin{eqnarray}\label{4.9}
    \int_{0}^{t}\int_{\mathbb{R}}e_{\theta}\mathbb{K}_{\theta}\chi_{t}^2\gtrsim
    \int_{0}^{t}\int_{\mathbb{R}}\left(1+a\theta^{3}\right)\left(1+\theta^{b}\right)\chi_{t}^{2}\gtrsim X(t)
\end{eqnarray}
and
\begin{eqnarray}\label{4.10}
    \int_{0}^{t}\int_{\mathbb{R}}\left(\mathbb{K}_{\theta}\chi_{x}\right)\left(\mathbb{K}_{\theta}\chi_{x}\right)_{t}
    &=&\frac{1}{2}\int_{\mathbb{R}}\left(\mathbb{K}_{\theta}\chi_{x}\right)^2(t,x)dx
    -\frac{1}{2}\int_{\mathbb{R}}\left(\mathbb{K}_{\theta}\chi_{x}\right)^2(0,x)dx\nonumber\\
    &\gtrsim& Y(t)-C.
\end{eqnarray}
We now estimate $I_{k}(k=11,12,\ldots,17)$ term by term. For the term $I_{11}$, it follows from Lemma 2.2, \eqref{2.7-delta} and \eqref{4.7} that
\begin{eqnarray}\label{4.11}
    I_{11}&\leq&\epsilon X(t)+C(\epsilon)\int_{0}^{t}\int_{\mathbb{R}}\left[\left(1+\theta^{3b-2}\right)\left|\Theta_{x}\right|^{4}
    +\left(1+\theta^{3b}\right)\left|\Theta_{xx}\right|^{2}
    +\left(1+\theta^{b}\right)\left|\Theta_{x}\right|^{2}\left|V_{x}\right|^{2}\right]\nonumber\\
    &\leq&\epsilon X(t)+C(\epsilon)\left[\left(1+M_{2}^{3b-2}\right)\delta+\left(1+M_{2}^{3b}\right)\delta^{\frac{1}{2}}
    +\left(1+M_{2}^{b}\right)\delta\right]\\
    &\leq&\epsilon X(t)+C(\epsilon).\nonumber
\end{eqnarray}
After simple calculation, we can deduce from Lemma 3.4 that
\begin{eqnarray}\label{4.12}
  \left|\frac{\theta p_{\theta}}{e_{\theta}}-\frac{\Theta P_{\Theta}}{E_{\Theta}}\right|\lesssim |\varphi|+|\chi|.
\end{eqnarray}
On the other hand, by using Taylor's formula, we can deduce for $0<\omega<1$ that
\begin{eqnarray}\label{4.12b}
 \int_{\mathbb{R}}\chi^{2}dx=2\int_{\mathbb{R}}\Phi\left(\frac{\theta}{\Theta}\right)\left(\omega\Theta+(1-\omega)\theta\right)^{2}dx
 \lesssim 1+M_{2}^{2}.
\end{eqnarray}
Thus we can obtain from Lemma 2.2, \eqref{2.7-delta}, \eqref{2.10}, \eqref{3.9}, \eqref{4.7}, \eqref{4.12}, and \eqref{4.12b} that
\begin{eqnarray}\label{4.13}
    I_{12}&\leq&\epsilon X(t)+C(\epsilon)\int_{0}^{t}\int_{\mathbb{R}}\left[\left(1+\theta^{b}\right)\left|\varphi_{x}\right|^{2}\left|\Theta_{x}\right|^{2}
    +\frac{\left(1+\theta^{b}\right)\Theta\left|\chi_{x}\right|^{2}}{v\theta^{^{2}}}\cdot\theta^{2}\left|V_{x}\right|^{2}
    +\left(1+\theta^{b+6}\right)\left(\varphi^{2}+\chi^{2}\right)\left|U_{x}\right|^{2}\right]\nonumber\\
    &&+C\left\|\Theta_{t}\right\|_{L^{\infty}}\int_{0}^{t}\int_{\mathbb{R}}
    \frac{\left(1+\theta^{b}\right)\Theta\left|\chi_{x}\right|^{2}} {v\theta^{^{2}}}\cdot\left(1+\theta^{b+1}\right)\\
    &\leq&\epsilon X(t)+C(\epsilon)\left[\left(1+M_{2}^{b+1}\right)\delta+M_{2}^{2}\delta^{4}+\left(1+M_{2}^{b+8}\right)\delta
    +\delta^{2}\left(1+M_{2}^{b+1}\right)\right]\nonumber\\
     &\leq&\epsilon X(t)+C(\epsilon).\nonumber
\end{eqnarray}

Moreover, we get by combining the estimates \eqref{2.10}, \eqref{4.2}, \eqref{4.3}, and \eqref{4.7} that
\begin{eqnarray}\label{4.14}
    I_{13}&\lesssim&\int_{0}^{t}\int_{\mathbb{R}}
    \frac{\left(1+\theta^{b}\right)\Theta\left|\chi_{x}\right|^{2}}{v\theta^{^{2}}}\cdot\left|\psi_{x}\right|\theta^{2}\lesssim
    \left(1+Y(t)^{\frac{2}{2b+3}}\right)\left(1+Z(t)^{\frac{3}{8}}\right)\nonumber\\
          &\leq& \epsilon Y(t)+C(\epsilon)\left(1+Z(t)^{\frac{6b+9}{16b+8}}\right).
\end{eqnarray}
By employing Sobolev's inequality, \eqref{3.9}, and \eqref{2.10}, we find that
\begin{eqnarray}\label{4.15}
    I_{14}&\leq&\epsilon X(t)+C(\epsilon)\int_{0}^{t}\left\|\frac{\kappa(v,\theta)\chi_{x}}{v}\right\|^{2}_{L^{\infty}}\|\varphi_{x}\|^{2}d\tau\nonumber\\
   &\leq&\epsilon X(t)+C(\epsilon)\left(1+\|\theta\|_{\infty}\right)
   \int_{0}^{t}\int_{\mathbb{R}}\left|\frac{\kappa(v,\theta)\chi_{x}}{v}\right|\left|\left(\frac{\kappa(v,\theta)\chi_{x}}{v}\right)_{x}\right|\\
   &\leq&\epsilon X(t)+C(\epsilon)\left(1+\|\theta\|_{\infty}\right)
   \left(\int_{0}^{t}\int_{\mathbb{R}}\theta^{2}\kappa\left|\left(\frac{\kappa(v,\theta)\chi_{x}}{v}\right)_{x}\right|^{2}\right)^{\frac{1}{2}}
\left(\int^{t}_{0}\int_{\mathbb{R}}\frac{\kappa(v,\theta)\Theta\chi^{2}_{x}}{v\theta^{2}}\right)^{\frac{1}{2}}\nonumber\\
 &\leq&\epsilon X(t)
    +C(\epsilon)\left(1+Y(t)^{\frac{1}{2b+3}}\right)
    \underbrace{\left(\int_{0}^{t}\int_{\mathbb{R}}\left(1+\theta^{b+2}\right)
    \left|\left(\frac{\kappa(v,\theta)\theta_{x}}{v}\right)_{x}
    -\left(\frac{\kappa(v,\theta)\Theta_{x}}{v}\right)_{x}\right|^{2}\right)^{\frac{1}{2}}}_{J^{\frac{1}{2}}}.\nonumber
\end{eqnarray}
And in view of $\eqref{2.2}_{3}$, Lemma 2.2, \eqref{2.7-delta}, \eqref{2.8}, \eqref{2.10}, \eqref{4.3}, and \eqref{4.12}, one has
\begin{eqnarray}\label{4.16}
   J&\lesssim&\int_{0}^{t}\int_{\mathbb{R}}\bigg[\left(1+\theta^{b+2}\right)e^{2}_{\theta}\chi^{2}_{t}
   +\left(1+\theta^{b+2}\right)\theta^{2}p^{2}_{\theta}\psi_{x}^{2}+\left(1+\theta^{b+2}\right)e^{2}_{\theta}\left(\frac{\theta p_{\theta}}{e_{\theta}}-\frac{\Theta P_{\Theta}}{E_{\Theta}}\right)^{2}U_{x}^{2}\nonumber\\ &&+\left(1+\theta^{b+2}\right)\psi_{x}^{4}
   +\left(1+\theta^{b+2}\right)U_{x}^{4}+\left(1+\theta^{b+2}\right)\phi^{2}z^{2}
   +\left(1+\theta^{b+2}\right)e^{2}_{\theta}r^{2}\left(V,\Theta\right)\bigg]\nonumber\\
   &&+\underbrace{\int_{0}^{t}\int_{\mathbb{R}}\left(1+\theta^{b+2}\right)
   \left|\left(\frac{\kappa(v,\theta)\Theta_{x}}{v}\right)_{x}\right|^{2}}_{J^{a}}\nonumber\\
   &\lesssim&J^{a}+\int_{0}^{t}\int_{\mathbb{R}}\left[\left(1+\theta^{b+8}\right)\chi^{2}_{t}+\frac{\mu\Theta\psi^{2}_{x}}{v\theta}
   \cdot\left(1+\theta^{b+11}\right)+\left(1+\theta^{b+8}\right)\left(\chi^{2}+\varphi^{2}\right)U^{2}_{x}\right]\\
   &&   +\left(1+\|\theta\|^{b+3}_{\infty}\right)\left\|\psi_{x}\right\|_{L^{\infty}}^{2}\int^{t}_{0}\int_{\mathbb{R}}
   \frac{\mu\Theta\psi^{2}_{x}}{v\theta}+\left(1+M^{b+8}_{2}\right)\delta^{\frac{4}{3}}\nonumber\\
   &&+\left(1+M^{b+2}_{2}\right)\delta+\left(1+\|\theta\|_{\infty}^{b+\beta+2}\right)\int_{0}^{t}\int_{\mathbb{R}}\phi z^{2}\nonumber\\
   &\lesssim& 1+X(t)\left(1+Y(t)^{^{\frac{8}{2b+3}}}\right)+Y(t)^{^{\frac{b+11}{2b+3}}}
   +\left(1+Y(t)^{\frac{b+3}{2b+3}}\right)\left(1+Z(t)^{\frac{3}{4}}\right) +Y(t)^{\frac{b+\beta+2}{2b+3}}+J^{a},\nonumber
\end{eqnarray}
and
\begin{eqnarray}\label{4.24}
  J^{a}&\lesssim&\int_{0}^{t}\int_{\mathbb{R}}\left(1+\theta^{b+2}\right)\bigg[\theta^{2b}\varphi^{2}_{x}\Theta^{2}_{x}+\theta^{2b}V^{2}_{x}\Theta^{2}_{x}
+\theta^{2b-2}\chi^{2}_{x}\Theta^{2}_{x}+\theta^{2b-2}\Theta^{4}_{x}+\Theta^{2}_{xx}+\theta^{2b}\Theta^{2}_{xx}\nonumber\\
&&+\left(1+\theta^{2b}\right)\Theta^{2}_{x}\left(\varphi_{x}^{2}+V^{2}_{x}\right)\bigg]\nonumber\\
&\lesssim&\left(1+\|\theta\|^{3b+2}_{\infty}\right)\int_{0}^{t}\left\|\Theta_{x}\right\|_{L^{\infty}}^{2}\left\|\varphi_{x}\right\|^{2}d\tau
+\left(1+\|\theta\|^{3b+2}_{\infty}\right)\int_{0}^{t}\int_{\mathbb{R}}V_{x}^{2}\Theta_{x}^{2}\\
&&+\int_{0}^{t}\int_{\mathbb{R}}\frac{\left(1+\theta^{b}\right)\Theta\left|\chi_{x}\right|^{2}}{v\theta^{^{2}}}
\cdot\left(1+\theta^{2b+2}\right)\Theta^{2}_{x}\nonumber\\
&&+\left(1+\|\theta\|^{3b}_{\infty}\right)\int_{0}^{t}\int_{\mathbb{R}}\Theta^{4}_{x}
+\left(1+\|\theta\|^{3b+2}_{\infty}\right) \int_{0}^{t}\int_{\mathbb{R}}\Theta^{2}_{xx}\nonumber\\
&\lesssim&\left(1+M_{2}^{3b+3}\right)\left(\delta^{\frac{1}{2}}+\delta^{4}\right)\lesssim 1.\nonumber
\end{eqnarray}

Thus we can conclude from \eqref{4.16}-\eqref{4.24} that
\begin{eqnarray}\label{4.25}
J &\lesssim& 1+X(t)\left(1+Y(t)^{^{\frac{8}{2b+3}}}\right)+Y(t)^{^{\frac{b+11}{2b+3}}}+Z(t)^{\frac{3}{4}}+Y(t)^{\frac{b+3}{2b+3}}Z(t)^{\frac{3}{4}}
+Y(t)^{\frac{b+\beta+2}{2b+3}}.
\end{eqnarray}
Plugging \eqref{4.25} into \eqref{4.15}, we have
\begin{eqnarray}\label{4.26}
I_{14}\leq\epsilon\left(X(t)+Y(t)\right)+C\left(\epsilon\right)\left(1+Z(t)^{\frac{6b+9}{12b+4}}\right).
\end{eqnarray}
Here we have used the fact that $b>\frac{19}{4}$ and $0\leq\beta<\min\{3b+2,5b-10\}$.

As for the term $I_{15}$, we have from Lemma 2.2, \eqref{2.9}, and the assumption $0\leq\beta<b+3$ that
\begin{eqnarray}\label{4.26a}
I_{15}&\lesssim&\int_{0}^{t}\int_{\mathbb{R}}\left[\left(1+\theta^{b}\right)\phi z+\left(1+\theta^{b}\right)
\left(1+a\theta^{3}\right)\left|r\left(V,\Theta\right)\right|\right]\left|\chi_{t}\right|\nonumber\\
&\leq&\epsilon X(t)+C\left(\epsilon\right)\int_{0}^{t}\int_{\mathbb{R}}\left[\left(1+\theta^{b}\right)\phi^{2}z^{2}+\left(1+\theta^{b+6}\right)
\left|r\left(V,\Theta\right)\right|^{2}\right]\\
&\leq&\epsilon X(t)+C\left(\epsilon\right)\left(1+\left\|\theta\right\|_{\infty}^{b+\beta}+\delta^{\frac{4}{3}}
\left(1+M_{2}^{b+6}\right)\right)\nonumber\\
&\leq&\epsilon\left(X(t)+Y(t)\right)+C\left(\epsilon\right).\nonumber
\end{eqnarray}
For the term $I_{16}$, we employ \eqref{4.7} and the assumption $b>6$ to find that
\begin{eqnarray}\label{4.26b}
I_{16}&\lesssim&\int_{0}^{t}\int_{\mathbb{R}}\left(1+\theta^{b}\right) \theta\left(1+a\theta^{3}\right)\left|\psi_{x}\chi_{t}\right|\nonumber\\
&\leq&\epsilon X(t)+C\left(\epsilon\right)\int_{0}^{t}\int_{\mathbb{R}}\left(1+\theta^{b+9}\right)\frac{\psi_{x}^{2}}{\theta}\\
&\leq&\epsilon X(t)+C\left(\epsilon\right)\left(1+Y(t)^{\frac{b+9}{2b+3}}\right)\nonumber\\
&\leq&\epsilon\left(X(t)+Y(t)\right)+C\left(\epsilon\right).\nonumber
\end{eqnarray}

It suffices to bound the term $I_{17}$. To this end, we conclude from Lemma 2.2 and \eqref{4.7} that
\begin{eqnarray}\label{4.26c}
I_{17}
&\leq&\epsilon X(t)+C\left(\epsilon\right)\int_{0}^{t}\int_{\mathbb{R}}\left(1+\theta^{b}\right)\left(\psi^{4}_{x}+U^{4}_{x}\right)
\leq\epsilon X(t)+C\left(\epsilon\right)\left(1+\int_{0}^{t}\int_{\mathbb{R}}\left(1+\theta^{b}\right)\psi^{4}_{x}\right).
\end{eqnarray}
Then by virtue of Sobolev's inequality, Lemma 3.6, and \eqref{2.10}, we can get
\begin{eqnarray}\label{4.26d}
\int_{0}^{t}\int_{\mathbb{R}}\left(1+\theta^{b}\right)\psi_{x}^{4}&\lesssim&
\left(1+\left\|\theta\right\|^{b}_{\infty}\right)\int_{0}^{t}\left\|\psi_{x}\right\|^{2}_{L^\infty}
     \left\|\psi_{x}\right\|^{2}d\tau\nonumber\\
     &\lesssim&
    \left(1+\left\|\theta\right\|^{b}_{\infty}\right)\int_{0}^{t}\left\|\psi_{x}\right\|^{3}
    \left\|\psi_{xx}\right\|d\tau\nonumber\\
    &\lesssim&
    \left(1+\left\|\theta\right\|^{b+3}_{\infty}\right)
    \left(\int_{0}^{t}\left\|\psi_{x}\right\|^{2}d\tau\right)^{\frac{1}{2}}
    \left(\int_{0}^{t}\left\|\psi_{xx}\right\|^{2}d\tau\right)^{\frac{1}{2}}\\
    &\lesssim& 1+\left\|\theta\right\|^{b+5}_{\infty}.\nonumber
\end{eqnarray}
Thus the combination of \eqref{4.26c}, \eqref{4.26d}, and the assumption $b>2$ gives
\begin{eqnarray}\label{4.27}
I_{17}\leq\epsilon\left(X(t)+Y(t)\right)+C\left(\epsilon\right).
\end{eqnarray}

Substituting \eqref{4.9}-\eqref{4.27} into \eqref{4.8} and by choosing $\epsilon>0$ small enough yield \eqref{4.4}.
\end{proof}

The next lemma tells us that $Z(t)$ can be controlled by $X(t)$ and $Y(t)$.
\begin{lemma} Under the conditions listed in Lemma 2.5, we have for all $0\leq t\leq T$ that
\begin{eqnarray}\label{4.30}
Z(t)\lesssim 1+X(t)+Y(t)+Z(t)^{\frac{6b+9}{8b+8}}.
\end{eqnarray}

\end{lemma}

\begin{proof}
We differentiate $\eqref{2.2}_{2}$ with respect to $t$ and multiply the result identity by $\psi_{t}$ to derive
\begin{eqnarray}\label{4.32}
      &&\left(\frac{\psi_{t}^2}{2}\right)_{t}+\frac{\mu\psi_{xt}^2}{v}
      +\left[\left(\left(p-P\right)_{t}+g(V,\Theta)_{t}-\mu\left(\frac{u_{x}}{v}\right)_{t}\right)\psi_{t}\right]_{x}\nonumber\\
      &=&\left[\frac{\mu\left(\psi_{x}+U_{x}\right)^2}{v^2}-\frac{\mu U_{xt}}{v}+\left(p-P\right)_{t}+g(V,\Theta)_{t}\right]\psi_{xt}.
  \end{eqnarray}
Integrating the above identity with respect to $t$ and $x$ over $(0,t)\times\mathbb{R}$, one has
\begin{eqnarray}\label{4.33}
     &&\int_{\mathbb{R}}\frac{\psi_{t}^{2}}{2}dx +\int_{0}^{t}\int_{\mathbb{R}}\frac{\mu\psi_{tx}^{2}}{v}\nonumber\\
     &=&    \underbrace{\int_{0}^{t}\int_{\mathbb{R}}\left(\frac{\mu\psi_{tx}\psi^{2}_{x}+\mu\psi_{tx}U^{2}_{x}+2\mu\psi_{tx}\psi_{x}U_{x}}{v^{2}}
     -\frac{\mu\psi_{tx}U_{tx}}{v}+g\left(V,\Theta\right)_{t}\psi_{tx}\right)}_{I_{18}}\\
     &&+\int_{\mathbb{R}}\frac{\psi_{0t}^{2}}{2}dx+\underbrace{\int_{0}^{t}\int_{\mathbb{R}}\left(p-P\right)_{t}\psi_{tx}}_{I_{19}}.\nonumber
  \end{eqnarray}
It suffices to estimate the terms $I_{k} (k=18, 19)$. For this purpose, we compute from \eqref{2.7-delta}, \eqref{2.10}, and \eqref{3.8} that
\begin{eqnarray}\label{4.34}
I_{18}&\leq&\epsilon \int_{0}^{t}\int_{\mathbb{R}}\frac{\mu\psi_{tx}^{2}}{v}+C\left(\epsilon\right)
\int_{0}^{t}\int_{\mathbb{R}}\left(\frac{\psi^{2}_{x}}{\theta}\cdot\left(\theta\psi^{2}_{x}+\theta U^{2}_{x}\right)
+U_{x}^{4}+U_{xt}^{2}+\left|g\left(V,\Theta\right)_{t}\right|^{2}\right)\nonumber\\
&\leq&\epsilon \int_{0}^{t}\int_{\mathbb{R}}\frac{\mu\psi_{tx}^{2}}{v}
+C\left(\epsilon\right)\left(\left(1+Y(t)^{\frac{1}{2b+3}}\right) \left(1+Z(t)^{\frac{3}{4}}\right)+\delta+\delta^{4}M_{2}\right)\\
&\leq&\epsilon \int_{0}^{t}\int_{\mathbb{R}}\frac{\mu\psi_{tx}^{2}}{v}
+C\left(\epsilon\right)\left(1+Y(t)+Z(t)^{\frac{6b+9}{8b+8}}\right).\nonumber
\end{eqnarray}
Moreover, it is easy to see that
\begin{eqnarray}\label{4.35}
\left|\left(p-P\right)_{t}\right|^{2}\lesssim\left(1+a^{2}\theta^{6}\right)\chi_{t}^{2}+\left|\Theta_{t}\right|^{2}
\left(\varphi^{2}+\chi^{2}\left(1+\theta^{4}\right)\right) +\chi^{2}\psi_{x}^{2}+\chi^{2}U_{x}^{2}+\psi_{x}^{2} +U_{x}^{2}\varphi^{2}.
\end{eqnarray}
Then it follows from Lema 2.2 and \eqref{2.10} that
\begin{eqnarray}\label{4.36}
I_{19}&\leq&\epsilon \int_{0}^{t}\int_{\mathbb{R}}\frac{\mu\psi_{tx}^{2}}{v}
+C\left(\epsilon\right)\int_{0}^{t}\int_{\mathbb{R}}\left|\left(p-P\right)_{t}\right|^{2}\nonumber\\
&\leq&\epsilon \int_{0}^{t}\int_{\mathbb{R}}\frac{\mu\psi_{tx}^{2}}{v}
+C\left(\epsilon\right)\bigg(X(t)+ \int_{0}^{t}\left\|\Theta_{t}\right\|^{2}_{L^{\infty}}\left(\left\|\varphi\right\|^{2}
+\left(1+\left\|\theta\right\|^{4}_{L^{\infty}}\right)\left\|\chi\right\|^{2}\right)d\tau\nonumber\\
&&+\int_{0}^{t}\int_{\mathbb{R}}\frac{\psi_{x}^{2}}{\theta}\cdot\left(1+\theta^{3}\right)+\int_{0}^{t}
\left\|U_{x}\right\|^{2}_{L^{\infty}}d\tau\bigg)\\
&\leq&\epsilon \int_{0}^{t}\int_{\mathbb{R}}\frac{\mu\psi_{tx}^{2}}{v}
+C\left(\epsilon\right)\left(1+X(t)+Y(t)+\delta\left(1+M^{6}_{2}\right)\right)\nonumber\\
&\leq&\epsilon \int_{0}^{t}\int_{\mathbb{R}}\frac{\mu\psi_{tx}^{2}}{v}
+C\left(\epsilon\right)\left(1+X(t)+Y(t)\right).\nonumber
\end{eqnarray}
Choosing $\epsilon>0$ small enough, the combination of \eqref{4.33}-\eqref{4.36} and \eqref{3.8} shows
\begin{eqnarray}\label{4.37}
     \|\psi_{t}\|^{2}+\int_{0}^{t}\left\|\psi_{tx}(\tau)\right\|^{2}d\tau\lesssim 1+X(t)+Y(t)+Z(t)^{\frac{6b+9}{8b+8}}.
  \end{eqnarray}
Now we are in a position to yield an estimate on $\|\psi_{xx}(t)\|$. Firstly, $\eqref{2.2}_{2}$ tells us that
\begin{eqnarray}\label{4.38}
\psi_{xx}=\frac{\varphi_{x}\psi_{x}+\varphi_{x}U_{x}+V_{x}\psi_{x}+V_{x}U_{x}}{v}-U_{xx}
+\frac{v}{\mu}\left[\psi_{t}+\left(p(v,\theta)-P(V,\Theta)\right)_{x}+g(V,\Theta)_{x}\right].
  \end{eqnarray}
On the other hand, Lemma 2.2, \eqref{2.7-delta}, \eqref{3.9}, \eqref{4.2}, and \eqref{4.12b} show that
\begin{eqnarray}\label{4.39}
    \left\|(p-P)_{x}\right\|^2
    &\lesssim&\int_{\mathbb{R}} \left[\left(1+a^{2}\theta^6\right)|\chi_{x}|^2+\left|\left(V_{x},\Theta_{x}\right)\right|^2|\varphi|^2
    +\left(1+\theta^4\right)|(V_{x},\Theta_{x})|^2|\chi|^2+|\theta\varphi_{x}|^2\right]dx\nonumber\\
    &\lesssim&Y(t)+\left(1+\|\theta\|_{L^{\infty}}^2\right)\left(1+\|\theta\|_{L^{\infty}}\right)
     +\delta^4\left(1+M_{2}^6\right)\\
     &\lesssim&1+Y(t).\nonumber
\end{eqnarray}

Thus one can conclude from Lemma 3.6, \eqref{3.9}, \eqref{4.3}, \eqref{4.37}-\eqref{4.39}, and Young's inequality that
\begin{eqnarray}\label{4.40}
   \int_{\mathbb{R}}\psi^{2}_{xx}dx&\lesssim& \int_{\mathbb{R}}\left(\varphi^{2}_{x}\psi^{2}_{x}+\varphi^{2}_{x}U^{2}_{x}
   +\psi_{x}^{2}V^{2}_{x}+V^{2}_{x}U_{x}^{2}+U_{xx}^{2}+\psi_{t}^{2}+\left|\left(p(v,\theta)-P(V,\Theta)\right)_{x}\right|^{2}
   +\left|\left(g(V,\Theta)\right)_{x}\right|^{2}\right)dx\nonumber\\
   &\lesssim& 1+X(t)+Y(t)+Z(t)^{\frac{6b+9}{8b+8}}+\left\|\psi_{x}\right\|^{2}_{L^{\infty}}\left\|\varphi_{x}\right\|^{2}
   +\left\|U_{x}\right\|^{2}_{L^{\infty}}\left\|\varphi_{x}\right\|^{2}+\left\|V_{x}\right\|^{2}_{L^{\infty}}\left\|\psi_{x}\right\|^{2}\nonumber\\
   &&+\left\|V_{x}\right\|^{2}_{L^{\infty}}\left\|U_{x}\right\|^{2}+\left\|U_{xx}\right\|^{2}+\left\|\left(g(V,\Theta)\right)_{x}\right\|^{2}\nonumber\\
   &\lesssim& 1+X(t)+Y(t)+Z(t)^{\frac{6b+9}{8b+8}}+\delta^{4}\left(1+M_{2}^{3}\right)
   +\left(1+Y(t)^{\frac{1}{2b+3}}\right)\left(1+Z(t)^{\frac{3}{4}}\right)\nonumber\\
   &\lesssim&1+X(t)+Y(t)+Z(t)^{\frac{6b+9}{8b+8}}.
\end{eqnarray}
We thus get the estimate \eqref{4.30} by using the definition of $Z(t)$.
\end{proof}

We can deduce that $Y(t)\lesssim 1$ by combining Lemmas 4.1-4.3. Then the desired upper bound on the absolute temperature $\theta\left(t,x\right)$ follows from \eqref{4.2} immediately. Moreover, we can infer from Lemma 2.1-4.3 that

\begin{lemma} Under the conditions listed in Lemma 2.5, there exists a positive constant $C_{2}$ which depends only on $\underline{V}$, $\overline{V}$, $\underline{\Theta}$, $\overline{\Theta}$, and $H_0$, such that
  \begin{eqnarray}\label{4.41}
 \theta\left(t,x\right) \leq C_{2},\quad \forall
 \left(t,x\right)\in[0,T] \times\mathbb{R}.
  \end{eqnarray}
Moreover, we have for $0\leq t\leq T$ that
\begin{eqnarray}\label{4.42}
&&\sup\limits_{0\leq t<\infty}\left\|\left(\varphi, \psi, \chi, z, \varphi_{x}, \psi_{x}, \psi_{t}, \chi_{x}, \psi_{xx}\right)(t)\right\|^{2}\nonumber\\
&&+\int_{0}^{t}\left\|\left(\sqrt{\theta}\varphi_{x}, \psi_{x}, \chi_{t}, \chi_{x}, \psi_{xx}, \psi_{xt}, z_{x}\right)(\tau)\right\|^{2}d\tau\\
&\lesssim& 1.\nonumber
\end{eqnarray}
and
\begin{eqnarray}\label{4.43}
 \int_{0}^{t}\left\|\psi_{x}(\tau)\right\|^4_{L^4(\mathbb{R})}d\tau\lesssim 1,\quad \|\psi_{x}\|_{L^\infty([0,T]\times\mathbb{R})}\lesssim 1.
\end{eqnarray}
\end{lemma}

The next lemma gives nice bounds on the terms $\int_{0}^{t}\left\|\chi_{xx}(\tau)\right\|^{2}d\tau$ and $\left\|z_{x}(t)\right\|^{2}$, whose proof is similar to Lemma 4.5 developed in \cite{Liao-NARWA-2020}. Thus we omit the proof for brevity.

\begin{lemma} Under the conditions listed in Lemma 2.5, we can get for $0\leq t\leq T$ that
\begin{eqnarray}\label{5.7}
\left\|\chi_{x}(t)\right\|^{2}+\int_{0}^{t}\left\|\chi_{xx}(\tau)\right\|^2d\tau\lesssim1,
\end{eqnarray}
and
\begin{eqnarray}\label{5.13}
\left\|z_{x}(t)\right\|^{2}+\int_{0}^{t}\left\|z_{xx}(\tau)\right\|^2d\tau\lesssim 1.
\end{eqnarray}
\end{lemma}

\section{A local-in-time lower bound on the absolute temperature}
The following lemma will give a local-in-time lower bound on $\theta(t,x)$. In fact, we can deduce by repeating the argument developed in \cite{Liao-Zhao-JDE-2018} that

\begin{lemma} Under the conditions stated in Lemma 2.5, we have for each $0\leq s\leq t\leq T$ and $x\in\mathbb{R}$ that the following estimate
  \begin{equation}\label{5.16}
  \theta\left(t,x\right)\geq \frac{C\min\limits_{x\in\mathbb{R}}\{\theta(s,x)\}}{1+(t-s)\min\limits_{x\in\mathbb{R}}\{\theta(s,x)\}}
  \end{equation}
holds for some positive constant $C$ which depends only on $\underline{V}$, $\overline{V}$, $\underline{\Theta}$, $\overline{\Theta}$, and $H_0$.
\end{lemma}

\section{The Proof of main results}

With the above preparations in hand, we now turn to prove our main results.

We first prove Theorem 1.1. To this end, suppose that $(\varphi(t,x), \psi(t,x), \chi(t,x), z(t,x))\in X(0,T;M_1,M_2)$  is a solution to the Cauchy problem \eqref{2.2} and \eqref{2.3} defined on the strip $\Pi_T:=[0,T]\times\mathbb{R}$ and satisfying the \emph{a priori} assumption \eqref{2.6}, then if the assumptions listed in Theorem 1.1 hold true and $\delta>0$ and $a>0$ are chosen sufficiently small such that \eqref{2.7-delta} and \eqref{2.7-a} hold, we can get from Lemma 2.6, Lemma 3.4, Lemma 4.4, and Lemma 5.1 that
\begin{eqnarray}\label{A priori Estimates}
0\leq z(t,x)&\leq &1,\quad \forall (t,x)\in[0,T]\times\mathbb{R},\nonumber\\
C_1^{-1}\leq v(t,x)&\leq& C_1,\quad  \forall (t,x)\in[0,T]\times\mathbb{R},\nonumber\\
\theta(t,x)&\leq &C_2,\quad \forall (t,x)\in[0,T]\times\mathbb{R},\\
\theta(t,x)&\geq& \frac{C_3\min\limits_{x\in\mathbb{R}} \{\theta(s,x)\}}{1+(t-s)\min\limits_{x\in\mathbb{R}}\{\theta(s,x)\}},\quad  \forall (t,x)\in [s,t]\times\mathbb{R}\nonumber
\end{eqnarray}
hold for some positive constants $C_i (i=1,2,3)$ which depend only on $\underline{V}$, $\overline{V}$, $\underline{\Theta}$, $\overline{\Theta}$, and $H_0$.

Having obtained \eqref{A priori Estimates}, Theorem 1.1 can be proved by combining the local solvability result Lemma 2.3 with the continuation argument introduced in \cite{Liao-Zhao-JDE-2018,Wang-Zhao-M3AS-2016} and we omit the details for brevity.

Recall that in the proof of Theorem 1.1, the assumption on the smallness of the radiation constant $a$ is only used in Lemma 2.4 to guarantee that $\tilde{p}(v,s)$ is convex with respect to $v$ and $s$, and we do not use such a smallness assumption elsewhere to control certain nonlinear terms involved. As explained in the introduction, the very reason for such an analysis is that once we can imposed some other assumptions to
guarantee the convexity of $\tilde{p}(v,s)$ with respect to $(v,s)$ in the regime for $v$ and $s$ under our consideration, then one can deduce that similar stability result holds accordingly.

The main purpose of Theorem 1.2 is to show that if we use the smallness of $a$ to control the involved nonlinear terms, then we can relax the assumptions we imposed on the parameters $b$ and $\beta$ while similar stability result still holds.
For this purpose, we only need to re-estimate those terms related to the radiation constant $a$, since the terms can be estimated in the same way as in the proof of  Theorem 1.1.

First of all, we treat the term $\left\|\varphi_{x}(t)\right\|^2$. By using \eqref{2.7-a}, \eqref{2.10}, \eqref{2.20}, and \eqref{3.8}, $I_{6}$ can be re-estimated as
\begin{eqnarray}\label{7.1}
 I_{6}&\leq&\frac{1}{10}\int^{t}_{0}\int_{\mathbb{R}}\frac{R\theta\varphi^{2}_{x}}{v^{3}}
 +C\left(\left\|\theta\right\|_{\infty}+
 \left\|\frac{\theta^{2}\left(\frac{R}{v}+\frac{4a\theta^{3}}{3}\right)^{2}}
 {\kappa(v,\theta)p_{v}\left(v,\theta\right)}\right\|_{\infty}\right)\nonumber\\
&\leq&\frac{1}{10}\int^{t}_{0}\int_{\mathbb{R}}\frac{R\theta\varphi^{2}_{x}}{v^{3}}
 +C\left(1+\left\|\theta\right\|_{\infty}\right).
\end{eqnarray}
Inserting \eqref{7.1}, \eqref{2.22}-\eqref{2.25} into \eqref{2.19} and employing \eqref{3.8}, we can infer that
\begin{eqnarray}\label{7.2}
\left\|\varphi_{x}(t)\right\|^2+\int_{0}^{t}\int_{\mathbb{R}}\theta\varphi_{x}^{2}\lesssim 1 +\|\theta\|_{\infty}.
\end{eqnarray}

Now we deal with the term $\int_{0}^{t}\|\psi_{xx}(\tau)\|^{2}d\tau$. By virtue of \eqref{2.7-a} and \eqref{2.10}, we have
\begin{eqnarray}\label{7.3}
\int_{0}^{t}\int_{\mathbb{R}}\left(1+a^{2}\theta^6\right)|\chi_{x}|^2\lesssim
\int^{t}_{0}\int_{\mathbb{R}}\frac{\kappa(v,\theta)\Theta\chi^{2}_{x}}{v\theta^{2}}\cdot\frac{\theta^{2}}{1+\theta^{b}}\lesssim
1+\|\theta\|^{(2-b)_{+}}_{\infty}
  \end{eqnarray}
Plugging \eqref{7.3} into \eqref{a5.3} and utilizing \eqref{7.2}, we deduce that
\begin{eqnarray}\label{7.4}
\left\|\psi_{x}(t)\right\|^{2}+\int_{0}^{t}\int_{\mathbb{R}}\frac{\mu\psi_{xx}^{2}}{v}
&\leq&2\epsilon \int_{0}^{t}\int_{\mathbb{R}}\frac{\mu\psi_{xx}^{2}}{v}
+C\left(\epsilon\right)\left(1+\|\theta\|^{(2-b)_{+}}_{\infty}+\|\theta\|^{3}_{\infty}\right)\nonumber\\
&\leq&2\epsilon \int_{0}^{t}\int_{\mathbb{R}}\frac{\mu\psi_{xx}^{2}}{v}
+C\left(\epsilon\right)\left(1+\|\theta\|^{3}_{\infty}\right).
  \end{eqnarray}
By choosing $\epsilon>0$ small enough, we can see \eqref{a5.1} still holds true without imposing any condition on the parameter $b$.

On the other hand, \eqref{2.7-a} tells us that
\begin{eqnarray}\label{7.5}
\int_{0}^{t}\int_{\mathbb{R}}\left(1+\theta^{b+2}\right)e^{2}_{\theta}\chi^{2}_{t}\lesssim \int_{0}^{t}\int_{\mathbb{R}}\left(1+\theta^{b+2}\right)\chi^{2}_{t}\lesssim X(t)\left(1+Y(t)^{^{\frac{2}{2b+3}}}\right)
\end{eqnarray}
and
\begin{eqnarray}\label{7.6}
\int_{0}^{t}\int_{\mathbb{R}}\left(1+\theta^{b+2}\right)\theta^{2}p^{2}_{\theta}\psi_{x}^{2}\lesssim \int_{0}^{t}\int_{\mathbb{R}}\frac{\mu\Theta\psi^{2}_{x}}{v\theta}\cdot\left(1+\theta^{b+5}\right)\lesssim 1+Y(t)^{^{\frac{b+5}{2b+3}}}.
\end{eqnarray}
Then \eqref{4.16}, \eqref{4.24}, \eqref{7.5}, and \eqref{7.6} imply that
\begin{eqnarray}\label{7.7}
J &\lesssim& 1+X(t)\left(1+Y(t)^{^{\frac{2}{2b+3}}}\right)+Y(t)^{\frac{b+5}{2b+3}}
+Z(t)^{\frac{3}{4}}+Y(t)^{\frac{b+3}{2b+3}}Z(t)^{\frac{3}{4}}
+Y(t)^{\frac{b+\beta+2}{2b+3}}.
\end{eqnarray}
We utilize \eqref{4.15}, \eqref{7.7}, the assumption $b>\frac{5}{6}$, and $0\leq\beta< 3b+2$ to derive \eqref{4.26}.

Meanwhile, it follows from \eqref{2.7-a}, \eqref{4.7}, and the fact $b>0$ that

\begin{eqnarray}\label{7.8}
I_{16}&\lesssim&\int_{0}^{t}\int_{\mathbb{R}}\left(1+\theta^{b+1}\right) \left|\psi_{x}\chi_{t}\right|\nonumber\\
&\leq&\epsilon X(t)+C\left(\epsilon\right) \int_{0}^{t}\int_{\mathbb{R}}\left(1+\theta^{b+3}\right) \frac{\psi_{x}^{2}}{\theta}\\
&\leq&\epsilon X(t)+C\left(\epsilon\right)\left(1+Y(t)^{\frac{b+3}{2b+3}}\right)\nonumber\\
&\leq&\epsilon\left(X(t)+Y(t)\right)+C\left(\epsilon\right).\nonumber
\end{eqnarray}

We can exploit the same method developed in Section 4 to estimate the other terms. Here we need the condition $0\leq\beta<b+3$ to bound the term $I_{15}$ and $b>2$ to bound the term $I_{17}$. By repeating the argument used to prove Theorem 1.1, we can complete the proof of Theorem 1.2.

\bigbreak
\section{Acknowledgement}
Guiqiong Gong was supported by the Fundamental Research Funds for the Central Universities and the grants from the National Natural Science Foundation of China under contracts 11731008, 11671309, and 11971359. Lin He was partially supported by the Fundamental Research Funds for the Central Universities No.YJ201962.
Yongkai Liao was supported by National Postdoctoral Program for Innovative Talents of China No. BX20180054. We would like to express our thanks to the anonymous referees for their valuable comments, which lead to substantial improvements of the original manuscript. Last but not least, the authors would like to thank Professor Huijiang Zhao for his support and encouragement.

\end{document}